\documentclass[11pt, reqno]{amsart}
\usepackage{graphicx}
\usepackage{amsmath,amssymb,amsthm,amsfonts,mathrsfs,amscd,amsopn}
\title{Pseudotropical curves}
\author{Sergei Lanzat}
\address{Department of Mathematics, Technion- Israel
Institute of Technology, Haifa 32000, Israel}
\email{lanzat.sergei@gmail.com}
\author{Michael Polyak}
\address{Department of Mathematics, Technion- Israel
Institute of Technology, Haifa 32000, Israel}
\email{polyak@math.technion.ac.il}

\newcommand{\FAT}[1]{\mbox{{$\mathbb{#1}$}}}
\newcommand{\Fat}[1]{\mbox{{$\scriptscriptstyle\mathbb{#1}$}}}

\newcommand{\CC}{\FAT{C}}
\newcommand{\PP}{\FAT{P}}
\newcommand{\QQ}{\FAT{Q}}
\newcommand{\RR}{\FAT{R}}
\renewcommand{\SS}{\mathbb{S}}
\renewcommand{\ss}{\scriptstyle\mathbb{S}}
\newcommand{\ZZ}{\FAT{Z}}
\newcommand{\cZ}{\mathcal{Z}}
\newcommand{\cD}{\mathcal{D}}

\newcommand{\rr}{\Fat{R}}

\newcommand{\cc}{\Fat{C}}
\newcommand{\FG}{\mathfrak{g}}
\newcommand{\eps}{\varepsilon}
\newcommand{\G}{\Gamma}
\newcommand{\GL}{\Lambda}

\newcommand{\I}{\scriptscriptstyle I}
\newcommand{\II}{\scriptscriptstyle II}
\newcommand{\minus}{\smallsetminus}
\newcommand{\PTC}{plane pseudotropical curve}
\newcommand{\PTCs}{plane pseudotropical curves}

\newcommand\restrict[1]{\hspace{-0.7mm}\raisebox{-.5ex}{$|$}_{#1}}
\DeclareMathOperator{\codim}{codim}
\DeclareMathOperator{\ev}{ev}
\DeclareMathOperator{\mult}{mult}
\DeclareMathOperator{\sign}{sign}
\newtheorem{thm}{Theorem}[section]
\newtheorem{thm*}{Theorem}
\newtheorem{lem}[thm]{Lemma}
\newtheorem{prop}[thm]{Proposition}
\newtheorem{cor}[thm]{Corollary}
\newtheorem{rem}[thm]{Remark}

\theoremstyle{definition}
\newtheorem{defn}[thm]{Definition}
\newtheorem{defn*}[thm*]{Definition}

\newtheorem{ex}[thm]{Example}

\begin{document}

\begin{abstract}
We propose a generalization of tropical curves by dropping the rationality and integrality
requirements while preserving the balancing condition. An interpretation of such curves
as critical points of a certain quadratic functional allows us to settle the existence and
uniqueness problem. The machinery of dual polygons and the intersection theory also
generalize as expected. We study the homology of a compactified moduli space of rigid
oriented marked curves. A weighted count of rational pseudotropical curves passing
through a generic collection of points is interpreted via top-degree cycles on the moduli.
We construct a family of such cycles using quantum tori Lie algebras and show that in
the usual tropical case this gives the refined curve count of Block and G\"ottsche.
Finally, we derive a recursive formula for this Lie-weighted count of rational pseudotropical curves.
\end{abstract}

\thanks{Both authors were partially supported by the ISF grant 1794/14}

\maketitle

\section{Introduction and main definitions}
\label{sec: intro}

\subsection{Motivation}
Tropical geometry is a rapidly developing field of mathematics that uses combinatorial and
piecewise linear structures to study and solve problems in different fields such as optimization,
combinatorics, algebraic geometry, integrable and dynamical systems, applied mathematics,
economics, computational biology, etc.
For a brief introduction to tropical geometry see e.g. \cite{IMS,M-St,M2,RGST05} and
references within.
In particular, tropical geometry is a powerful tool for exploring algebraic varieties.
Algebraic varieties can be degenerated to tropical varieties, which are polyhedral complexes
satisfying certain combinatorial properties.
While they have a simple combinatorial structure, tropical varieties encode a lot of information
about the geometry of the initial varieties and allow one to translate complicated algebraic or
geometric problems into a simple combinatorial language.

The most classical and well-studied case of algebraic varieties is that of algebraic curves.
A tropical counterpart of plane algebraic curves are balanced rational metric graphs in the plane.
Loosely speaking, a plane tropical curve is a finite union of rays and segments, each endowed with
a slope vector whose coordinates are integer (and directions are rational) and a positive integer
multiplicity; the balancing condition means that at every vertex the sum of the outgoing slope
vectors counted with their multiplicities adds up to zero.

In a number of works coming from different areas we see the appearance of similar graphs, but with
arbitrary edge vectors without any integrality/rationality conditions on their slopes and coordinates.
In particular, in algebraic geometry objects of this nature had appeared in a number of different
problems, see e.g. \cite{BH}, \cite{E}, \cite{Kaz}, \cite{Kri}, \cite{Lang}, \cite{MS}.

We propose a unified treatment of such generalized objects, which we call pseudotropical curves,
by dropping the rationality and integrality conditions while preserving the balancing condition.
Apart from their direct applications to above mentioned problems (which cannot be treated by the
classical tropical technique requiring rationality/integrality), our approach offers a new viewpoint
and provides new tools for studying the standard tropical curves.
Firstly, separating the integrality conditions from the rest leads to a better understanding of the
underlying nature and core properties of tropical curves.
Secondly, relaxing the rationality and integrality conditions allows slope vectors to vary in continuous
or smooth families, enabling one to study the dependence of moduli and enumerative characteristics
of curves on these parameters (see e.g. Section \ref{subsec: eps-dependence}); we expect this to
lead to some interesting differential equations.
Thirdly, dealing with more general objects gives additional degrees of freedom, which gives rise to an
additional flexibility of the theory (e.g., it leads to new recursion relations in enumerative problems,
see Section \ref{sec: recursion}).

\subsection{Structure of the paper}
In the Subsection \ref{subsec: definitions} below we introduce the basic object of our study - the
notion of \PTCs. In Section \ref{sec: varcalculus} we interpret \PTCs\ as critical points of a certain
quadratic functional and study the existence and uniqueness problem in terms of solutions of the
discrete Laplace equation with the Neumann-type boundary condition. In Section \ref{sec: duality}
we define dual polygons, study the intersections of \PTCs\ and prove the B\'{e}zout's and Bernstein's
theorems. In Section \ref{sec: moduli} we define orientations of \PTCs, introduce a notion of rigid
marked curves and study moduli spaces of rational oriented \PTCs. In Section \ref{sec: homology}
we study the top homology of a compactified moduli space of marked \PTCs; we study an enumerative
problem of a weighted count of rational irreducible \PTCs\ passing through a generic collection of
points and interpret it as a generalized degree of the evaluation map.
In Section \ref{sec: Lie and enumeration} we show how to produce generalized degrees as above
in terms of quantum tori Lie algebras and study the corresponding enumerative problem. In particular,
in the classical tropical case we identify our weighted curve count with the refined count of Block and
G\"ottsche \cite{BG}. In Section \ref{sec: recursion} we derive a recursive formula for the above 
Lie-weighted count of rational irreducible \PTCs.

\subsection{Definition of \PTCs}
\label{subsec: definitions}
In this section we define the notion of \PTCs\ that generalizes that of plane tropical curves.
For the basic notions of tropical curves, see \cite{FM, GM1, GM2, M1, M2}.

Consider a simple finite graph (i.e., a finite 1-dimensional CW-complex without loops and multiple edges)
without bivalent vertices. Let $\G$ be its subset obtained by removing all univalent vertices.
We shall use the following notations:
\begin{itemize}
\item  $V$ -- the set of vertices of $\G$,
\item $E$ -- the set of edges of $\G$,
\item $E^\infty$ -- the set of \emph{legs}, i.e. (half-open) edges of $\G$ incident to univalent vertices,
\item $E^0=E\minus E^\infty$ -- the set of \emph{bounded} edges,
\item $\partial e$ -- the set of vertices incident to the edge $e$.
\end{itemize}
For convenience we will usually assume that the set of legs is ordered: $E^\infty=\{e_1,e_2,\dots,e_n\}$.
Abstract tropical curves are metric graphs, where lengths of some edges are allowed to be infinite:
\begin{defn}\label{def:  Abstract tropical curves}
An {\em abstract tropical curve} is a pair $(\G, \ell)$, where $\ell$ is a complete intrinsic metric,
defined by prescribing (positive, real) length $l_e$ to each bounded edge $e\in E^0$ and $+\infty$ to
each leg of $\G$.
A curve $(\G, \ell)$ is {\em irreducible of genus $g$} if $\G$ is connected
and $g = 1 - |V| + |E^0|$ is the first Betti number $b_1(\G)$ of $\G$.
\end{defn}

\begin{defn}\label{def:PTC}
A \emph{parameterized \PTC}\ is a triple $(\G, \ell, h)$ such that
\begin{itemize}
 \item[(i)] $(\G, \ell)$ is an abstract tropical curve and
\item[(ii)] $h:\G\to\RR^2$ is a continuous proper map satisfying the following properties:
  \begin{itemize}
        \item[(a)]
        The map $h$ is an affine map on each edge $e\in E$.
        Namely, let $v\in \partial e$; then there is a vector $\xi_v(e)\in T_{h(v)}\RR^2$ such
        that the restriction $h|_e$ of $h$ to $e$ is given by $h(v)+t\,\xi_v(e)$ with $t\in[0,l_e]$.
        We assume that $\xi_v(e)\ne 0$ for $e\in E^\infty$.
        \item[(b)]
        For a fixed vertex $v$, the vectors $\xi_v(e)$ satisfy the \emph{balancing condition}
        \begin{equation}
        \label{eq:balancing}
        \sum_{e:v\in\partial e} \xi_v(e)=0,
        \end{equation}
        Note that if $\partial e = \{v, v'\}$ then $\xi_{v}(e) = - \xi_{v'}(e)$.
  \end{itemize}
\end{itemize}

We shall often work with oriented edges, so that $\partial e=e^+ - e^-$, where $e\in E^0$
is oriented from the starting vertex $e^-$ towards the terminal vertex $e^+$.  In this case
we will define $\xi(e)$ as $\xi_{e^-}(e)$.
Note that $\xi(-e)=-\xi(e)$, where $-e$ is $e$ with the opposite orientation.
Also, throughout the paper we always orient all legs towards infinity.

Two parameterized \PTCs\ $(\G, \ell, h)$ and $(\G', \ell', h')$ are isomorphic
if there exists an isometry $\rho: (\G, \ell)\to (\G', \ell')$ such that $h = h'\circ\rho$.
A \emph{\PTC}\ is an isomorphism class $C$ of a parameterized \PTC\ $(\G, \ell, h)$.
A \PTC\ is {\em irreducible of genus $g$} if an underlying abstract curve is such.
An irreducible \PTC\ of genus zero will be called a \emph{rational \PTC}.

A {\em $\Delta$-set} of $C$ is the (ordered) set of vectors $\xi$ along the legs of $\G$, i.e.
$$\Delta(C)=\{\xi(e_1),\xi(e_2),\dots,\xi(e_n)\}\,.$$
Note that the fact that all vertices of $C$ are balanced implies that all vectors in the $\Delta$-set sum
up to zero: $\sum_{i=1}^n\xi(e_i)=0$.
It is convenient to think intuitively about legs as about edges towards an ``infinite'' vertex; the condition
that all vectors in the $\Delta$-set sum up to zero then means that this vertex is balanced.
\end{defn}

\begin{figure}[htb]
%\vspace{1.2in}
\includegraphics[width=5in]{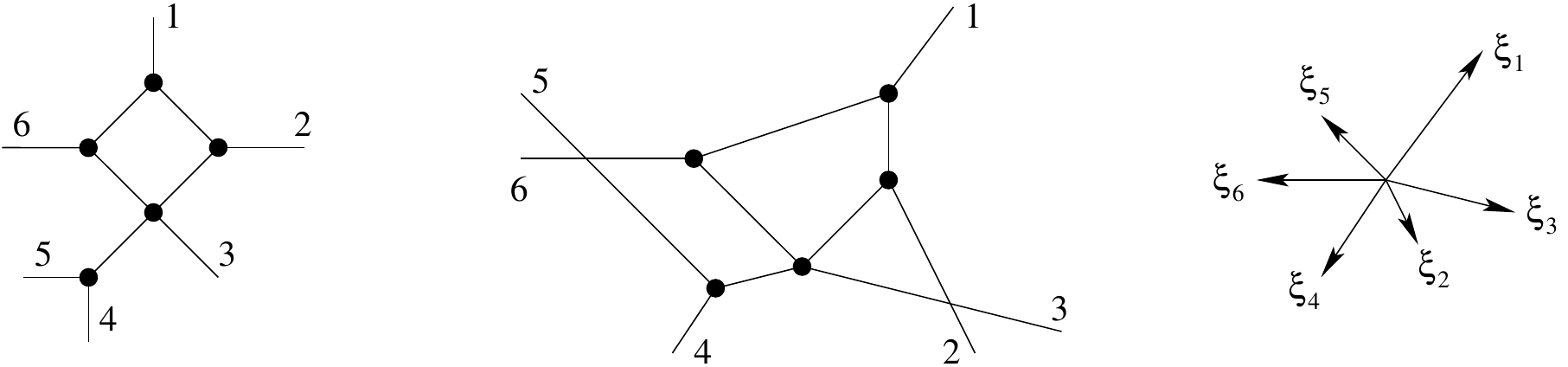}
\caption{An abstract curve, a plane pseudotropical curve and its $\Delta$-set}
\label{fig: pseudotropical curve}
\end{figure}

Throughout the paper we will identify $\RR^2$ with $\CC$.

\subsection*{Acknowledgements} We thank I. Itenberg, N. Kalinin, S. Kharlamov, G. Mikhalkin,
and E. Shustin for numerous discussions. A significant part of this work was done while the second
author was visiting CINVESTAV in Mexico City, whose hospitality he gratefully acknowledges.
Both authors were partially supported by the ISF grant 1794/14.

\section{Interrelations with variational calculus}
\label{sec: varcalculus}

\subsection{Plane pseudotropical curves as critical points}
\label{subsec: Neumann}
We shall interpret \PTCs\ (in particular, the balancing condition of Definition \ref{def:PTC}) as
critical points of a certain quadratic functional in a variational problem with boundary constrains.
The existence of such curves will be seen as a consequence of the existence and uniqueness of
a solution of the discrete Laplace equation with the Neumann-type boundary condition.

Fix an abstract tropical curve $(\G,\ell)$.
Assume that $\G$ is connected (otherwise consider the problem separately on every connected
component of $\G$) and all bounded edges of $\G$ are arbitrarily oriented.

We will use (with minor modifications) a standard terminology from electric networks (see e.g. \cite{BF}).
A {\em complex-valued potential} on $\G$ is a function $\phi:V\to\CC$.
A {\em boundary current} on $\G$ is a function $\xi:E^\infty\to\CC$.
Define the \emph{Neumann power functional} $P_\xi(\phi)\in\RR$ by
\begin{equation}\label{eq: Neumann power functional}
P_\xi(\phi) =
\sum_{e\in E^0}\frac{|\phi(e^+)-\phi(e^-)|^2}{2 l_e}
-\sum_{e\in E^\infty}\mathrm{Re}\left(\xi(e)\cdot\overline{\phi(e^-)}\right).
\end{equation}
We are looking for functions $\phi:V\to\CC$ that minimize the Neumann power functional $P_\xi(\phi)$
for a given boundary current $\xi$.
Given $\phi$, we extend $\xi$ to a {\em global current} $\widetilde{\xi}:E\to\CC$ by
$$\widetilde{\xi}(e) =\frac{\phi(e^+)-\phi(e^-)}{l_e}$$
for any $e\in E^0$ (so $\widetilde{\xi}(e)$ is the ``slope'', i.e. rate of change, of $\phi$ along $e$ per unit length).

\begin{prop}\label{prop: Neumann Minimum iff balanced}
A potential $\phi:V\to\CC$ minimizes the Neumann power functional $P_\xi(\phi)$ if and only if the
global current is balanced at every vertex, i.e., for any $v\in V$ we have
$$
\sum_{e:\ v=e^-} \widetilde{\xi}(e) = \sum_{e:\ v=e^+} \widetilde{\xi}(e)
$$
\end{prop}

\begin{proof}
Note that the quadratic part of $P_\xi$ (see \eqref{eq: Neumann power functional}) is a real-valued
nonnegative quadratic form on $\CC^V=(\RR^2)^V$.
It follows that $\phi$ minimizes $P_\xi$ if and only if $\phi$ is its critical point. Denote
$\phi = \phi_1 + i\phi_2$, $\xi = \xi_1 + i\xi_2$ and $\widetilde{\xi} = \widetilde{\xi_1} + i\widetilde{\xi_2}$.
Then we have
$$
P_\xi(\phi) = \sum_{i=1}^2\left(
\frac{(\phi_i(e^+)-\phi_i(e^-))^2 }{2 l_e}
-\sum_{e\in E^\infty}\xi_i(e)\cdot\phi_i(e^-)\right)
$$
It follows that for any $v\in V$ and $i=1,2$ we have
$$
\begin{aligned}
\frac{\delta P_\xi}{\delta \phi_i(v)} =
& \sum_{\substack{{e\in E^0:}\\{\ v=e^+}}}\frac{\phi_i(e^+)-\phi_i(e^-)}{l_e}
-\sum_{\substack{{e\in E^0:}\\{\ v=e^-}}}\frac{(\phi_i(e^+)-\phi_i(e^-)}{l_e}
-\sum_{\substack{{e\in E^\infty:}\\{\ v=e^-}}}\xi_i(e) = \\
=& \sum_{e\in E:\ v=e^+}\widetilde{\xi_i}(e) - \sum_{e\in E:\ v=e^-}\widetilde{\xi_i}(e)\,.
\end{aligned}
$$
This proves the statement.
\end{proof}

\subsection{Graph Laplacian}
\label{subsec: Laplace}
Let us reformulate the criticality condition in terms of a weighted discrete Laplace operator on $\G$.
Denote by $C^0\cong\CC^V$ and $C^1\cong\CC^E$ spaces of $0$-cochains and $1$-cochains
of $\G$ with coefficients in $\CC$, respectively.
Define $d:C^0\to C^1$ and $d^*:C^1\to C^0$ by
$$(d\phi)(e)=
\begin{cases}
\frac{\phi(e^+)-\phi(e^-)}{l_e}, & \text{if}\ e\in E^0\\
0,\ & \text{if}\ e\in E^\infty
\end{cases}
\ \ \quad
(d^*\psi)(v)=\sum_{\substack{{e\in E:}\\{\ v=e^+}}}\psi(e) - \sum_{\substack{{e\in E:}\\{\ v=e^-}}}\psi(e).
$$
Also, consider a zero extension of $\xi$ from a function on $E^\infty$ to a 1-cochain in $C^1$ by
setting $\xi(e)=0$ for any $e\in E^0$; by abuse of notation denote this cochain by $\xi\in C^1$.
Consider a weighted discrete Laplace operator $\nabla^2:=d^*d:C^0\to C^0$.
The proof of Proposition \ref{prop: Neumann Minimum iff balanced} implies that $\phi\in C^0$
minimizes the Neumann power functional $P_\xi$ if and only if
\begin{equation}\label{eq:Discrete Neumann Problem}
\nabla^2(\phi) = - d^*\xi.
\end{equation}

We can consider $\xi$ as a normal derivative of $\phi$ along legs at boundary vertices
(i.e., vertices incident to legs). Therefore, \eqref{eq:Discrete Neumann Problem} can be seen
as a discrete version of a classical Neumann problem (explaining, in particular, the term
``Neumann power functional'').

In a basis of $C^0$ given by some ordering $V=\{v_1,\cdots, v_k\}$ of the set of vertices, the matrix of
$\nabla^2$ is equal to the usual weighted Laplacian matrix $\GL$ of $\G$, see \cite{GR,Baker-Faber}.
Its elements $\GL_{ij}$ are given by

$$
\GL_{ij} = \begin{cases}
\sum\limits_{e:\ v_i\in\partial e}l_e^{-1},\ &\text{if}\ i=j \\
-l_e^{-1},\ &\text{if}\ \partial e =\{v_i,v_j\} \\
0, &\text{otherwise}.
\end{cases}
$$
The matrix representation of $- d^*\xi$ in the same basis of $C^0$ is given by a column vector
$b$ with $b_i = \sum\xi(e)$, where the sum is over all $e\in E^\infty$ with $e^-= v_i$.
Thus the matrix representation of \eqref{eq:Discrete Neumann Problem} is given by
$$\GL\phi=b\,,$$
where $\phi$ is a column vector with entries $\phi(v_i)$.

\subsection{Existence and uniqueness}
\label{subsec: Existence and uniqueness}
Next, we shall prove the existence and uniqueness theorem for the discrete Neumann problem.

\begin{prop}\label{prop: Existence and uniqueness for the Neumann problem}
Let $(\G,\ell)$ be an abstract tropical curve. Given a boundary current $\xi$ on $\G$, there exists
a solution $\phi\in C^0$ to the discrete Neumann problem \eqref{eq:Discrete Neumann Problem}
if and only if $\sum_{e\in E^\infty}\xi(e) =0$.
Any two solutions differ by a constant function on $V$.
\end{prop}
\begin{proof}
We are looking for solutions of the linear system $\GL\phi=b$.
Note that by definition of $\GL$ the sum of coefficients in every row and column equals zero.
Therefore, if there is a solution to the system we have that $\sum_i b_i=0$, hence $\sum_{e\in E^\infty}\xi(e) =0$.

The opposite direction and the uniqueness follow from a well-known fact\footnote{See e.g.
\cite[Lemma 13.1.1]{GR} in the case of the standard Laplacian matrix; the same proof works
in the weighted case. For an alternative proof see \cite{Baker-Faber}.} that for a connected
graph $\G$ the kernel of $\GL$ is 1-dimensional and is spanned by $(1,\dots,1)^t$.
\end{proof}

\begin{cor}\label{cor: Uniqueness of extension}
For any abstract tropical curve $(\G,\ell)$ and a boundary current $\xi:E^\infty\to\CC\minus\{0\}$
with $\sum_{e\in E^\infty}\xi(e) =0$, there exists a corresponding \PTC\ $C$ with $\Delta$-set
$\Delta(C)=\{\xi(e_i)\}_{i=1}^n$. Such a curve $C$ is unique up to a translation in $\RR^2$.
For an irreducible abstract tropical curve $(\G,\ell)$ of genus zero the corresponding slope vectors
$\{\widetilde{\xi}(e)\}_{e\in E^0}$ on bounded edges do not depend on the metric $\ell$.
\end{cor}
\begin{proof}
Indeed, Proposition \ref{prop: Existence and uniqueness for the Neumann problem} implies that
there exists a solution $\phi$ to the discrete Neumann problem \eqref{eq:Discrete Neumann Problem}.
This solution defines a parametrization of $(\G,\ell)$ in the following way: for any vertex $v\in V$
we define $h(v)=\phi(v)\in\CC=\RR^2$. The balancing conditions for the resulting curve readily follow
from Proposition \ref{prop: Neumann Minimum iff balanced}. Translations of this curve in $\RR^2$
correspond to additions of constant functions to $\phi$, so the uniqueness follows from Proposition
\ref{prop: Existence and uniqueness for the Neumann problem}.

Now, let $(\G,\ell)$ be an irreducible abstract tropical curve of genus zero.
For $|E^0|=0$ the last statement is trivial. For $|E^0|>0$, note that there is always a vertex
$v$ with a unique adjacent bounded edge $e\in E^0$ with $v=e^+$. Indeed, remove all (open)
legs of $\G$; the remaining graph is a tree and hence has a leaf; take it to be $v$. The balancing
condition forces a unique extension $\widetilde{\xi}(e)$ of the boundary current $\xi$ to this edge.
Remove $v$ with all adjacent legs and make $e$ into a leg with the slope $\widetilde{\xi}(e)$.
The statement follows by induction on the number $|E^0|$ of bounded edges of $\G$.
\end{proof}

\subsection{Star-mesh transform}\label{susec: Star mesh}
The theory of electrical networks also suggests various interesting transformations of pseudotropical curves.
Let us briefly describe some of them leaving details to an interested reader.

Firstly, one can consider a more general class of underlying graphs, allowing bivalent vertices and/or multiple
edges. Namely, let $(\G,\ell,h)$ be a parameterized \PTC\ and $e$ be its bounded edge of length $l$.
Then one can replace $e$ by a pair of edges as shown in Figure \ref{fig: mult2valent}a (respectively,
Figure \ref{fig: mult2valent}b) with lengths $l'$, $l''$ so that $l=l'+l''$
($l^{-1}=l'\,^{-1}+l''\,^{-1}$, respectively).
The parametrization $h$ remains unchanged in case of a bivalent vertex of Figure  \ref{fig: mult2valent}a.
In case of a double edge of Figure \ref{fig: mult2valent}b, the parametrization $h$ remains unchanged on
all vertices and slopes $\xi_v(e')$, $\xi_v(e'')$ of the two new edges are defined by
$l'\cdot\xi_v(e')=l''\cdot\xi_v(e'')=l\cdot\xi_v(e)$.
Bivalent vertices can also be considered as marked points (see Section \ref{subsec: markedcurves}).

\begin{figure}[htb]
%\vspace{1in}
\includegraphics[width=5.0in]{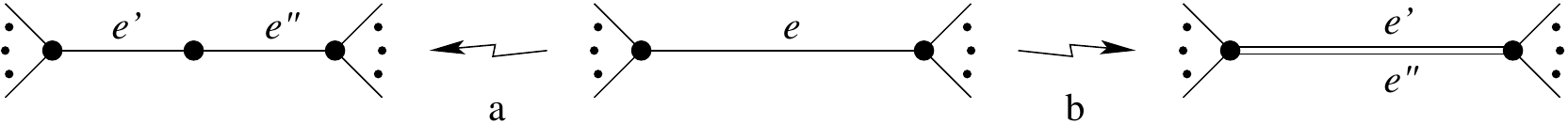}
\caption{Graphs with bivalent vertices and multiple edges.
\label{fig: mult2valent}}
\end{figure}

Secondly, one can apply the following celebrated {\em star-mesh transform}, see Figure \ref{fig: starmesh}a.
Let $(\G,\ell,h)$ be a parameterized \PTC\ and $v$ be its vertex not incident to legs. Let $v_i$, $i=1,\dots,s$
be the set of its neighboring vertices and $l_{i}$, $i=1,\dots,s$ be the set of lengths of corresponding edges
$e_{0i}$ connecting $v$ with $v_i$.
Replace the star of $v$ by a complete graph $K_s$ on the vertices $v_i$, $i=1,\dots,s$ with lengths of new
edges $e_{ij}$, connecting $v_i$ with $v_j$, given by
$\displaystyle l_{ij}=\frac{l_i\cdot l_j}{(l_1+l_2\dots+l_s)}$.
The parametrization $h$ remains unchanged on all vertices (except $v$), and slopes $\xi_{v_i}(e_{ij})$ are
defined by $l_{ij}\cdot\xi_{v_i}(e_{ij})=h(v_j)-h(v_i)$.
The proof follows from the Schur complement identity applied to the Laplace matrix $\GL$ of $\G$.

\begin{figure}[htb]
%\vspace{1in}
\includegraphics[width=5.0in]{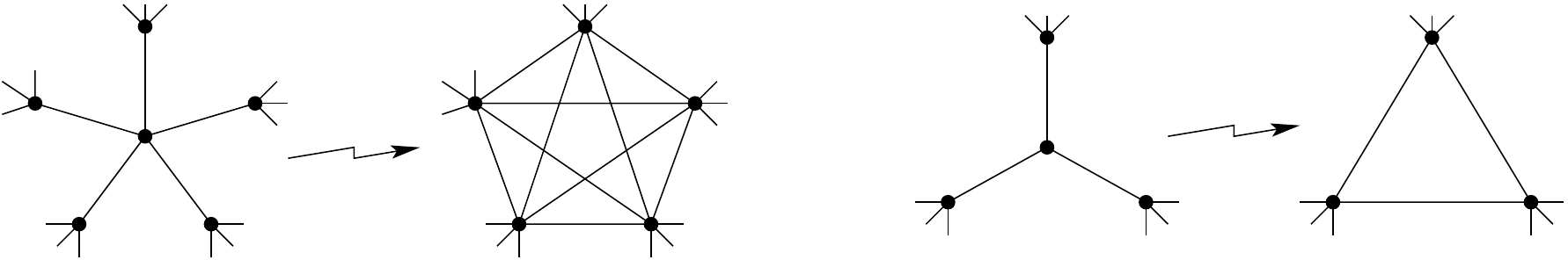}\\
\ a\hspace{2.7in}b
\caption{Star-mesh and $Y$-$\Delta$ transforms.
\label{fig: starmesh}}
\end{figure}

In particular, for $s=2$ we recover the transformation of a bivalent vertex considered above.
For $s=3$ this transformation is known also as a $Y$-$\Delta$ transform, see Figure \ref{fig: starmesh}b.
It preserves the number of bounded edges and thus has an inverse, known as a $\Delta$-$Y$ transform.
The corresponding formulas are
$$l_i=\frac{l_{ij}\cdot l_{ik}}{l_{12}+l_{23}+l_{31}}\,,\qquad
\xi_{v_i}(e_{0i})=\xi_{v_i}(e_{ij})+\xi_{v_i}(e_{ik})\,,$$
where $(ijk)$ is a cyclic permutation of $(123)$.
For $s>3$ the star-mesh transformation is not invertible in general without additional constraints on values
of $l_{ij}$.

\section{Dual polygons and intersections of pseudotropical curves}
\label{sec: duality}

\subsection{Dual polygon and \PTCs}
\label{subsec: duality}
We shall generalize to a pseudotropical case the notion of a dual polygon with the dual subdivision
corresponding to a tropical curve.
Let $C = [(\G, \ell, h)]$ be a \PTC\ with a $\Delta$-set $\Delta(C)$.
We orient all legs of $\G$ towards infinity and orient all bounded edges of $\G$ arbitrarily,
see Section \ref{subsec: definitions} for conventions.
In the case when the immersion $h:\G\to\RR^2$ is sufficiently general, i.e. self-intersection points
of $h(\G)$ are isolated, let $G$ be a graph, obtained from $h(\G)$ by considering all nodes
(i.e., self-intersection points) of $h(\G)$ as vertices.
New vertices subdivide corresponding edges of $h(\G)$ so that the resulting graph is planar.
Orientations, slope vectors $\xi(e)$ and lengths of new edges are inherited from the original
edges in an obvious way.
Since new vertices are balanced, $G$ can be considered as a \PTC\ with the same $\Delta$-set
and mapping $h$ (but possibly of a higher genus). See Figure \ref{fig: pinched graph}.

\begin{figure}[htb]
%\vspace{1in}
\includegraphics[height=1.2in]{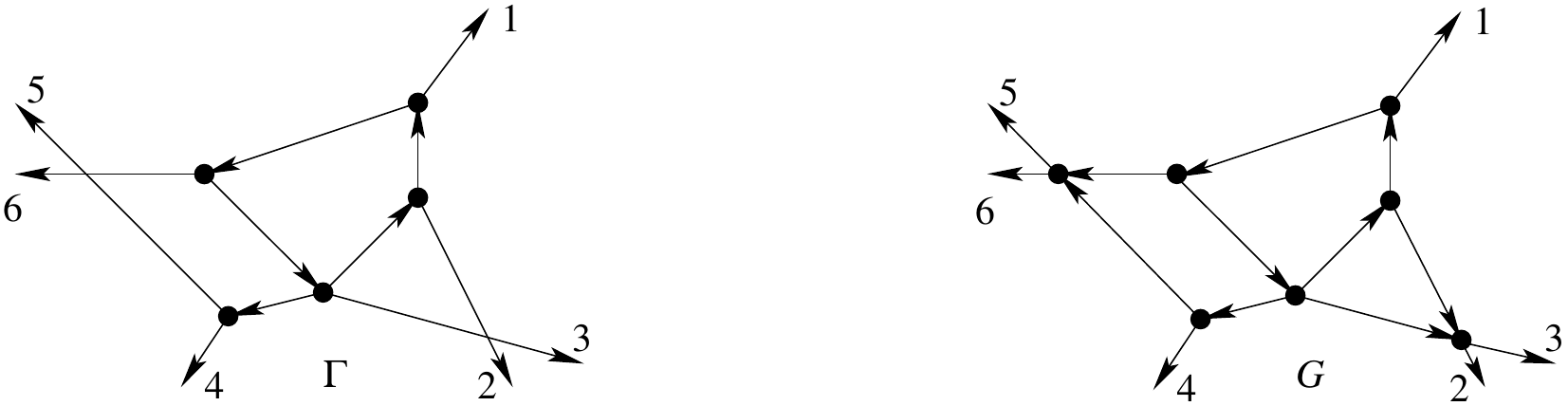}
\caption{Transforming nodes into vertices.
\label{fig: pinched graph}}
\end{figure}

In the degenerate case when the immersion $h:\G\to\RR^2$ has non-isolated self-intersections, a similar
construction of $G$ requires additional steps.
Non-isolated self-intersections appear when images of edges/legs intersect along intervals, see Figure
\ref{fig: degenerate immersion}a.
In this case to construct the graph $G$ we treat isolated self-intersection points as above, transforming them
to vertices; as for self-intersections along intervals, we apply transformations of addition/removal of 2-valent
vertices and merging of multiple edges into single edges, as described in Section \ref{susec: Star mesh}
(see Figure \ref{fig: mult2valent}). See Figure \ref{fig: degenerate immersion}b.  Edges of the resulting graph
$G$ are equipped with orientations, slope vectors $\xi(e)$ and lengths according to formulas of Section
\ref{susec: Star mesh}, so that all vertices of $G$ are balanced.

\begin{figure}[htb]
%\vspace{1in}
\includegraphics[width=5in]{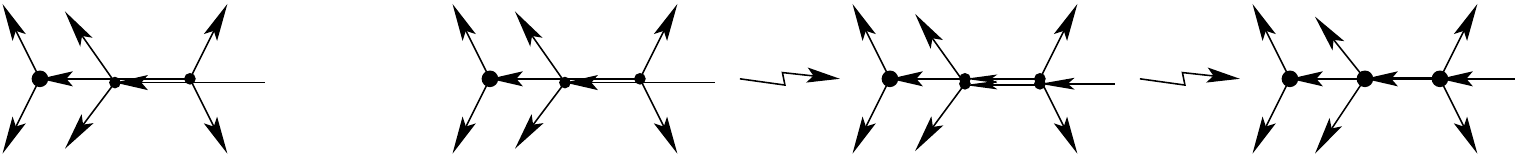}
\caption{Degenerate immersions.
\label{fig: degenerate immersion}}
\end{figure}

Denote by $G^*$ a planar graph dual to $G$, i.e. a graph obtained by placing a vertex in each region of $G$
(i.e., a connected component of $\RR^2\minus G$) and, if closures of two regions of $G$ share a common
edge $e$, connecting the corresponding vertices of $G^*$ by an edge $e^*$ which intersects only $e$.
Regions of $G^*$ correspond to vertices of $G$.
A problem of realizing $G^*$ by a {\em reciprocal} of $G$, i.e., a graph with straight edges orthogonal to
these of $G$ is well-studied; the standard approach involves the Maxwell-Cremona lifting. Since in our case
$G$ is already realized as a balanced (a.k.a. {\em stressed}) graph, the procedure is quite simple.

We define a dual (complete) metric $\ell^*$ on $G^*$ by setting $l^*_{e^*}=l_e^{-1}$ for every bounded
edge $e$ of $G$ and $l^*_{e^*}=1$ for every leg $e$ of $G$.
Also, we equip every edge $e^*$ of $G^*$ with a dual slope vector $\xi^*(e^*)$
obtained from $\xi$ by rotating it counterclockwise by $90$-degrees and rescaling:
$\xi^*(e^*):= (l^*_{e^*})^{-1}\cdot\sqrt{-1}\cdot\xi(e)$ (so that
$\xi^*(e^*)= l_e\cdot\sqrt{-1}\cdot\xi(e)$ for $e\in E^0$ and
$\xi^*(e^*)=\sqrt{-1}\cdot\xi(e)$ for $e\in E^\infty$).
Note that the metric $\ell^*$ and the rescaling factor in the definition of $\xi^*$ are chosen so, that
the resulting collection of slope vectors satisfies the balancing condition at every vertex of $G^*$ which
is dual to a bounded region.

We shall define a map $h^*:G^*\to\CC\cong\RR^2$ as follows.
Fix a spanning rooted tree $T$ in $G^*$.
Define $h^*$ arbitrarily (say, by 0) in the root vertex $v^*_0$.
For any other vertex $v^*$ of $G^*$ there exists a unique 1-chain $C(v^*)=\sum_{e^*\in C(v)} e^*$ of
oriented edges $e^*$ of $T$ such that $\partial C(v^*)=v^*-v^*_0$. Define
$$h^*(v^*)=\sum_{e^*\in C(v^*)} l^*_{e^*}\cdot \xi^*(e^*)=
\sum_{e^*\in C(v^*)} \sqrt{-1}\cdot\xi(e)\,.$$
Finally, extend $h^*$ to an affine map $h^*:G^*\to\RR^2$.
Note that the image $h^*(G^*)$ does not depend on a metric $\ell$ on $\G$.
See Figure \ref{fig: Newton polygon}.

\begin{figure}[htb]
%\vspace{1in}
\includegraphics[width=5in]{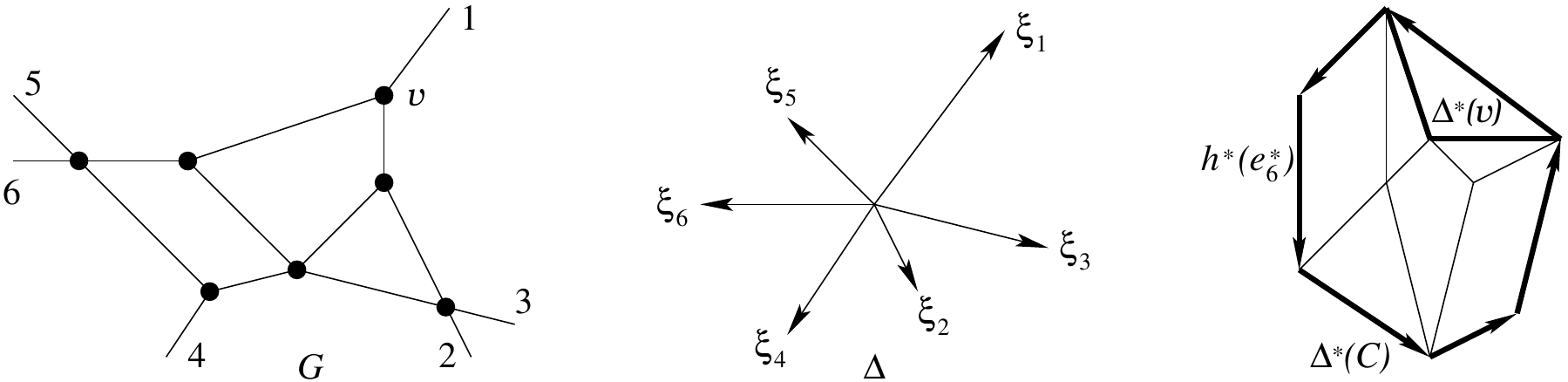}
\caption{Newton polygon $h^*(G^*)$.
\label{fig: Newton polygon}}
\end{figure}

\begin{prop}\label{prop: injectivity of dual parametrization}
The map $h^*:G^*\to\RR^2$ does not depend on the choice of a rooted tree $T$ up to a translation
in $\RR^2$.
Its image $h^*(G^*)$ forms a convex polygon with a convex subdivision.
\end{prop}

\begin{proof}
Recall that we have $\xi(-e)=-\xi(e)$ (and thus $\xi^*(-e^*)=-\xi^*(e^*)$), so when we move the root of
$T$ to a neighboring vertex along an edge $e^*$ of $G^*$ the map $h^*$ changes by an additive constant
$\pm\sqrt{-1}\cdot\xi(e)$, i.e., by a translation in $\RR^2$.
To prove an independence of $h^*$ on a choice of the tree $T$ it suffices to prove that for any closed
oriented path $Z$ of edges in $G^*$ one has $\sum\limits_{e^*\in Z} l^*_{e^*}\cdot \xi^*(e^*)=0$.
But cycles in $G^*$ are generated by boundaries of regions of the graph $G^*$; for such a positively
(i.e., counterclockwise) oriented cycle $Z(v)$ along the boundary of the region $R(v)$ dual to a vertex $v$
of $G$ we have
$$\sum_{e^*\in Z(v)}  l^*_{e^*}\cdot \xi^*(e^*) =\sum_{e^*\in Z(v)} \sqrt{-1}\cdot\xi(e)
=\sqrt{-1}\cdot \sum_{e: v\in\partial e} \xi_v(e)=0 $$
due to the balancing condition in the vertex $v$ of $G$.
Moreover, the image under $h^*$ of the  boundary of the region $R(v)$ forms a closed convex polygon;
indeed, as we follow a sequence of edges $e^*$ along $\partial R(v)$ counterclockwise, arguments of
the corresponding vectors $\xi_v(e)$ of dual edges meeting in the vertex $v$ (and hence arguments of
$\xi^*(e^*)$) are non-decreasing in $S^1$.
Denote such polygon (with a counterclockwise orientation, inherited from the positive orientation of
$\partial R(v)$) by $\Delta^*(v)$, see Figure \ref{fig: Newton polygon}.

Similarly, note that the image under $h^*$ of the ``outer perimeter'' of $G^*$ (i.e. the set of all vertices
of $G^*$ dual to unbounded regions of $G$ and edges $\{ e_i^*\}_{i=1}^m$ connecting them) form
a closed convex polygon.
Indeed, up to translations, images $h^*(e_i^*)$ are obtained from vectors $\xi(e_i)$,
$i=1,2,\dots,m$ of $\Delta(C)$ simply by a $90$-degrees counterclockwise rotation.
Since $\sum_{i=1}^m\xi(e_i)=0$, segments $h^*(e_i^*)$ form a unique (up to translations) closed
convex polygon. Note also, that since all legs are oriented towards infinity, this polygon inherits the
counterclockwise orientation. Denote it by $\Delta^*(C)$. See Figure \ref{fig: Newton polygon}.

It suffices to prove that each point $p$ in $\RR^2\minus h^*(G^*)$ is covered by at most one such
polygon $\Delta^*(v)$, so interiors of different polygons $\Delta^*(v)$ do not intersect and define a
convex subdivision of $\Delta^*(C)$.
Obviously, if a point - let's denote it by $\infty$ - is chosen sufficiently far from $h^*(G^*)$ then it is
not covered by any such polygon $\Delta^*(v)$.
Since all $\Delta^*(v)$ are counterclockwise oriented, the number of polygons covering $p$ equals
to the intersection number $I\left(\sum\limits_{v}\Delta^*(v), [p]-[\infty]\right)$ of the 2-chain
$\sum_{v}\Delta^*(v)$ with the 0-chain $[p]-[\infty]$, where the summation is over all vertices of $G$.
Alternatively, this number can be calculated as the intersection number
$I\left(\partial(\sum\limits_{v}\Delta^*(v)), -[p,\infty]\right)$ of the 1-chain
$\partial(\sum_{v}\Delta^*(v))$ with a 1-chain $-[p,\infty]$ represented by a generic simple path connecting
$p$ and $\infty$. Note that for any bounded edge $e$ of $G$ with $\partial e = e^+-e^-$ the corresponding
dual edge $e^*$ enters in the boundary $\partial(\Delta^*(e^{\pm}))$ of two neighboring polygons
$\Delta^*(e^{\pm})$ with opposite orientations, so $\partial(\sum_{v}\Delta^*(v))=\partial\Delta^*(C)$.
Thus $$I(\sum_{v}\Delta^*(v), [p]-[\infty])=I(\partial\Delta^*(C), -[p,\infty])\in\{0,1\}$$ depending on
whether $p$ lies outside or inside $\Delta^*(C)$ and the statement follows.
\end{proof}

In the pseudotropical setting, the convex polygon $\Delta^*(C)$ together with its subdivision as in the
proof above (see Figure \ref{fig: Newton polygon}) plays the role of the Newton polygon with a subdivision
corresponding to a tropical curve in a toric surface.
Also, it is easy to define a degree of a \PTC\ that recovers the degree of algebraic curves in $\PP^2$,
i.e. curves with the Newton polygon that is the convex hull of the set $\{(0,0), (d,0), (0,d)\}$.
However, unlike in the standard tropical setting, the degree $\deg(C) $ may not be an integer:

\begin{defn}\label{def: dual polygon and degree}
Let $C = [(\G, \ell, h)]$ be a \PTC\ with a $\Delta$-set $\Delta(C)$. A convex polygon
$\Delta^*(C)$ bounded by the outer perimeter of $h^*(G^*)$ is called the dual polygon of $C$.
The convex subdivision of $\Delta^*(C)$ given by $\Delta^*(v)$ for all vertices $v$ of $G$ is
called the dual subdivision of $C$. Note that
\begin{equation}
\label{eq:area of Delta}
Area(\Delta^*(C))=\sum_{v}  Area(\Delta^*(v))\,,
\end{equation}
where the summation is over all vertices of $G$ and $Area(A)$ is the Euclidean area of $A$.
The \emph{degree} $\deg(C)$ of $C$ is
\begin{equation}
\label{eq:degree of a pseudotropical curve}
\deg(C) :=\big(2 Area(\Delta^*(C))\big)^\frac12=\Big(2 \sum_{v}  Area(\Delta^*(v))\Big)^\frac12\,.
\end{equation}
\end{defn}

\subsection{Bernstein's and B\'{e}zout's Theorems for \PTCs}
Classical theorems of algebraic geometry have their tropical analogues.
Many of them can be generalized to the case of \PTCs\ in a straightforward manner.
Below we provide generalizations of tropical B\'{e}zout's and Bernstein's theorems
(see \cite[Theorem 4.6.8]{M-St} and \cite[Theorem 9.5]{St}).

\begin{defn}\label{def: Local intersection multiplicity}
Let $C_1 = [(\G_1, \ell_1, h_1)]$ and $C_2 = [(\G_2, \ell_2, h_2)]$ be two \PTCs.
We say that $C_1$ and $C_2$ are in {\em general position} with respect to each other, if $h_1(\G_1)$ and $h_2(\G_2)$ intersect each other transversally in inner points of edges.
Let $p\in h_1(\G_1)\cap h_2(\G_2)$ be an intersection point of two edges $h_1(e_1)$ and $h_2(e_2)$.
Define the local intersection multiplicity $I_p(C_1,C_2)$ of $C_1$ and $C_2$ in $p$ as

\begin{equation}
\label{eq:Local intersection multiplicity}
I_p(C_1,C_2) = |\det(\xi(e_1), \xi(e_2))|
\end{equation}
and the {\em total intersection multiplicity} of $C_1$ and $C_2$ as
$$I(C_1,C_2)=\sum_{p\in h_1(\G_1)\cap h_2(\G_2)}I_p(C_1,C_2)\,.$$
\end{defn}

Recall the notions of Minkowski sum and mixed area.
Let $A,B\subset\RR^2$ be two compact convex sets. The {\em Minkowski sum} $A\boxplus B$ of $A$ 
and $B$ is a compact convex set in $\RR^2$ defined by $$A\boxplus B = \{a + b|a\in A, b\in B\}.$$
In particular, if both $A$ and $B$ are convex polygons, $A\boxplus B$ is a convex polygon with the 
set of edges obtained by merging sets of edges of $A$ and $B$.
The {\em mixed area} $Area(A,B)$ of $A$ and $B$ is given by
\begin{equation}\label{eq: mixed area}
Area(A,B)= \frac12(Area(A\boxplus B)-Area(A)-Area(B)).
\end{equation}

\begin{thm}\label{thm: Bersnstein-Bezout Theorem}
Let $C_1 = [(\G_1, \ell_1, h_1)]$ and $C_2 = [(\G_2, \ell_2, h_2)]$ be two \PTCs\ in general position.
Then
$$I(C_1,C_2) = 2 Area(\Delta^*(C_1), \Delta^*(C_2))\ge\deg(C_1)\cdot\deg(C_2)\,.$$
Moreover, $I(C_1,C_2) = \deg(C_1)\cdot\deg(C_2)$ if and only if the dual polygons $\Delta^*(C_1)$ and
$\Delta^*(C_2)$ are homothetic\footnote{Here we consider polygons $\Delta^*(C_1)$ and $\Delta^*(C_2)$
as the convex hull of their points.}.
\end{thm}

\begin{proof}
Let $C_1\sqcup C_2=[(\G_1\sqcup \G_2, \ell_1\sqcup\ell_2, h_1\sqcup h_2)]$ be a (reducible) \PTC.
Since $\Delta(C_1\sqcup C_2)=\Delta(C_1)\sqcup\Delta(C_2)$, the dual polygon $\Delta^*(C_1\sqcup C_2)$
is the Minkowski sum $\Delta^*(C_1)\boxplus\Delta^*(C_2)$ of two dual polygons $\Delta^*(C_1)$ and
$\Delta^*(C_2)$.
From \eqref{eq: mixed area} and \eqref{eq:area of Delta} it follows that the doubled mixed area
$2 Area(\Delta^*(C_1), \Delta^*(C_2))$ is summed up from the areas of parallelograms
$\Delta^*(p)\subset\Delta^*(C_1\sqcup C_2)$ dual to the intersection points $p\in h_1(\G_1)\cap h_2(\G_2)$. 
Since $Area(\Delta^*(p)) = I_p(C_1,C_2)$ by \eqref{eq:Local intersection multiplicity}, the equality
$I(C_1,C_2) = 2 Area(\Delta^*(C_1), \Delta^*(C_2))$
follows.
The rest follows from Minkowski's first inequality\footnote{Minkowski's first inequality follows from
the celebrated Brunn-Minkowski theorem, see \cite[Theorem 7.1.1]{Sch}; its higher dimensional
generalization is the Alexandrov-Fenchel inequality, see \cite[Theorem 7.3.1]{Sch}} which states that
$$
Area(A,B)\geq \big(Area(A)\cdot Area(B)\big)^\frac12
$$
and the equality holds if and only if $A$ and $B$ are homothetic, see \cite[Theorem 7.2.1]{Sch}.

\end{proof}

\begin{cor}
If the dual polygons of \PTCs\ $C_1$ and $C_2$ are homothetic, then
$\deg(C_1\sqcup C_2) = \deg(C_1)+\deg(C_2)$.
\end{cor}

\section{Marked curves, their orientations and moduli}
\label{sec: moduli}
In this section we define orientations of \PTCs\ and introduce a notion of rigid marked curves.
We then study moduli spaces of rational oriented \PTCs\ for a suitable number $m=|\Delta|-1$
of marked points and analyze the corresponding evaluation maps.

\subsection{Moduli of \PTCs\ and their orientations}
\label{subsec: Moduli and orientation}
We say that two \PTCs\ $C=[(\G, \ell, h)]$ and $C'=[(\G', \ell', h')]$ with equal $\Delta$-sets have
the same \emph{combinatorial type}, if $\G$ is homeomorphic to $\G'$ (preserving the ordering of legs).
Corollary \ref{cor: Uniqueness of extension} implies that the set of lengths $\{l_e\}_{e\in E^0}$
and a point $h(v_*)\in\RR^2$ for a fixed {\em root} vertex $v_*\in V$  defines a curve $C$ uniquely.
It follows that lengths $\{l_e\}_{e\in E^0}$ and a point $h(v_*)\in\RR^2$ are free coordinates on a set
$\mathcal{M}_{\Delta}(\mu)$ of \PTCs\ $C=[(\G, \ell, h)]$ of a combinatorial type $\mu$.

In particular, after an ordering of $E^0$ and a choice of a root vertex $v_*$, the set
$\mathcal{M}_{\Delta}(\mu)$ can be identified with an (unbounded) open convex polyhedron
\begin{equation}\label{eq: Mu polyhedron}
\mathcal{M}_{\Delta}(\mu)\cong (\RR_+)^{|E^0|}\times\RR^2
\end{equation}
 in a real vector space $\RR^{|E^0|}\times\RR^2$.
A permutation of $E^0$ and a different choice $v_*'$ of a root vertex lead to an obvious linear
coordinate change on $(\RR_+)^{|E^0|}\times\RR^2$ by the corresponding reordering of the
coordinates on $(\RR_+)^{|E^0|}$ and a translation by the vector $h(v_*)-h(v_*')$ on $\RR^2$.
In particular, the Jacobian of such a coordinate change equals to $\pm1$ depending on the
parity of the permutation.
We define an orientation of $\mathcal{M}_{\Delta}(\mu)$ in an obvious way, using the
identification \eqref{eq: Mu polyhedron} above\footnote{It coincides with Kontsevich's
orientation of a graph complex in an even-dimensional case, see \cite{K}.}.
Namely, we choose an orientation of the vector space $\RR^{E^0}$, i.e., a sign-ordering
of the set $E^0$ of internal edges and complete it by the standard orientation of $\RR^2$.
Here by a {\em sign-ordering} of a set we mean an ordering up to even permutations.

Denote $\codim(\mu):= \sum\limits_{v\in V}(\mathrm{val}(v)-3)\,.$
Then we have $$|E^0|=|E^\infty|+3g-3 - \codim(\mu)=|\Delta|+3g-3- \codim(\mu)\,,$$
where $g=g(\G)$ is the genus of a graph $\G$ of type $\mu$.
So the set $\mathcal{M}_{\Delta}(\mu)$ can be also identified with
$(\RR_+)^{|\Delta|+3g-3-\codim(\mu)}\times\RR^2$.
Note that for a fixed $|\Delta|$ and $g$ top-dimensional polyhedra $\mathcal{M}_{\Delta}(\mu)$ correspond
to combinatorial types $\mu$ with $\codim(\mu)\nobreak=\nobreak0$, i.e., to combinatorial types of 3-valent
graphs. In particular, for $g=0$ top-dimensional polyhedra correspond to combinatorial types of 3-valent trees.
\begin{rem}
In case of $3$-valent graphs the number $|\Delta|+3g-3$ is the complex dimension of the moduli space of
Riemann surfaces of genus $g$ with $|\Delta|$ punctures, as expected from the theory of harmonic amoebas
of \cite{Kri} and their tropicalization (see \cite{Lang}).
Note that the corresponding moduli space of genuine plane tropical curves (with a fixed $\Delta$-set of integral
vectors) is of dimension $|\Delta|+g-1$ (see \cite[Proposition 2.23]{M1}). Thus, while for $g=0$
dimensions of these spaces coincide, for $g>0$ these spaces are quite different; in the tropical case parameters
$\{l_e\}_{e\in E^0}$ satisfy $2g$ constraints and fail to be free coordinates.
\end{rem}
In the rest of the paper we consider only curves of genus zero, leaving the higher genus case for a
future study.

Moduli spaces $\mathcal{M}_{\Delta}(\mu)$ corresponding to different combinatorial types of rational
curves with $n$ legs can be glued together to produce a grand moduli space of curves.
Its top-dimensional homology has a nice structure and can be identified with a certain quotient space
of trivalent trees (see Definition \ref{def: Jacobi} of Jacobi space $J_n$ in Section \ref{sec: homology}).
However, having in mind an enumerative geometry, we construct below a refined moduli space of rigid
marked curves with the same homology (but a different polyhedral structure), so that our enumerative
problems may be restated in the language of evaluation maps and the intersection theory, see
Theorem \ref{thm: count}.

\subsection{Moduli of marked rational curves}
\label{subsec: markedcurves}
We modify definitions of Sections \ref{subsec: definitions} and \ref{subsec: Moduli and orientation}
to include marked points.

Let $0<m<n$ and let $(\G, \ell)$ be an abstract rational tropical curve with the set
$E^\infty=\{e_1,\dots,e_{n+m}\}$ of legs.
We shall call the last $m$ legs $e_{n+i}$, $i=1,\dots,m$ the {\em marked legs} of $\G$
and denote by $E^m$ the set  $\{e_{n+1},\dots,e_{n+m}\}$ of marked legs.
Also, call the corresponding vertices $z_i=e_{n+i}^-$, $i=1,\dots,m$ the {\em marked vertices}
(or {\em marked points}) of $\G$.

Similairly to Section \ref{subsec: Moduli and orientation}, we define
\begin{defn}
\label{def: marked_curves}
A tuple $(\G, \ell, h, m)$ is called a \emph{parameterized $m$-marked curve} if $(\G, \ell, h)$
is a parameterized \PTC\ with an exception that every marked leg $e_{n+i}$ is contracted by
$h$ to a point $h(z_i)$, i.e. $\xi(e_{n+i}) = 0$ for all $i=1,\dots,m$.
\end{defn}

In figures we draw marked legs as dashed and marked vertices $z_i$ as fat points.

Two parameterized $m$-marked curves $(\G, \ell, h, m)$ and $(\G', \ell', h', m)$ have the same
combinatorial type, if there is an isomorphism $\rho:\G\to\G'$ (preserving markings and the
ordering of legs).
Moreover, these curves are called isomorphic, if $\rho$ above is an isometry such that $h = h'\circ\rho$.
The isomorphism class $C$ of  $(\G, \ell, h, m)$ is called a (rational) {\em $m$-marked curve}.
A curve $C=[(\G,\ell,h,m)]$ is {\em non-degenerate}, if for every unmarked vertex $v$ of $\G$
vectors $\xi_v(e)$ span $\RR^2$.

A {\em $\Delta$-set} $\Delta(C)$ of an $m$-marked curve $C =[(\G, \ell, h, m)]$ is the
(ordered) set of vectors $\xi$ along unmarked legs of $\G$, i.e.,
$$\Delta(C):= \{\xi(e_1),\dots,\xi(e_n)\}\,.$$
All other notions and statements of Sections \ref{sec: intro}--\ref{sec: duality} extend to
the case of marked curves.
Notably, Proposition \ref{prop: Existence and uniqueness for the Neumann problem} holds for
a general boundary current $\xi$ (in particular, when $\xi(e)=0$ for some $e\in E^\infty$),
so Corollary \ref{cor: Uniqueness of extension} holds also for marked curves.
Thus every marked curve $C=[(\G, \ell, h, m)]$ of a fixed combinatorial type $\mu$ is
uniquely defined by the lengths $\{l_e\}_{e\in E^0}$ of its bounded edges together with
an image $h(v_*)\in\RR^2$ of a root vertex $v_*\in V$.
Moreover, slope vectors $\{\xi(e)\}_{e\in E^0}$ do not depend on lengths
$\{l_e\}_{e\in E^0}$, but only on $\Delta$ and the combinatorial type $\mu$ of $C$
(see Corollary \ref{cor: Uniqueness of extension}).

We are interested only in some special types of markings:

\begin{defn}\label{def: rigid curves}
An $m$-marked curve $C =[(\G, \ell, h, m)]$ is called {\em rigid}, if each connected
component of $\G\minus E^m$ is a tree with at most one leg.
\end{defn}

\begin{ex}
For $n=2$ and $m=1$ there is only one type of curves, see Figure \ref{fig: rigid curves}a; it is rigid.
For $n=3$ and $m=1$ there are three types of non-rigid curves (which differ by the labels on the legs)
with two trivalent vertices and one rigid type with a 4-valent vertex. See Figure \ref{fig: rigid curves}b.
For $n=3$ and $m=2$ there are 22 types of curves: 12 types (which differ by the labels on the legs)
with 3 trivalent vertices, 9 types with one trivalent (either marked or unmarked) vertex and one 4-valent
vertex, and one type with a 5-valent vertex.
Thirteen rigid types are shown in Figure \ref{fig: rigid curves}c, and 9 non-rigid types are shown in
Figure \ref{fig: rigid curves}d.
\end{ex}

\begin{figure}[htb]
%\vspace{1in}
\includegraphics[width=5in]{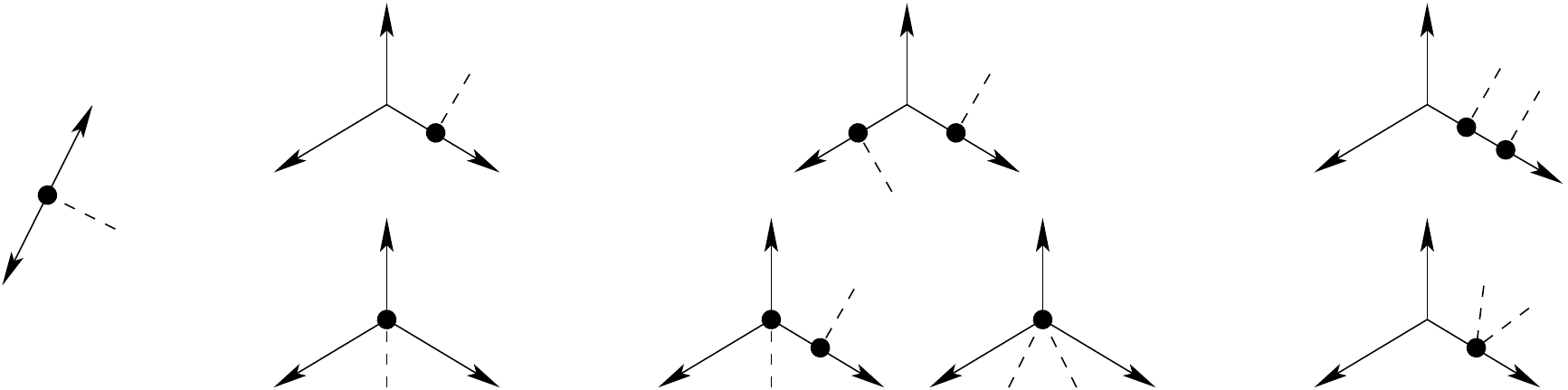}\\
a\hspace{1.0in}b\hspace{1.55in}c\hspace{1.6in}d\hspace{0.25in}
\caption{Examples of rigid and non-rigid marked curves.
\label{fig: rigid curves}}
\end{figure}

Note that since we assume that $m < n$, the rigidity condition gives a non-trivial restriction
on the maximal possible number of bounded edges of rigid trees with $n$ unmarked and
$m$ marked legs.
Indeed, a straightforward check shows that while for $m \ge n-1$ the maximal number of bounded
edges is $m+n-3$ (attained on trivalent trees as expected), for $m \le n-1$ this number is $2m-2$
and trees with a larger number of bounded edges fail to be rigid.

\begin{rem}
\label{rem: if m greater than n-1}
Our initial assumption $m<n$ on the number $m$ of marked points will become more clear
when we shall consider the evaluation map, mapping a curve into images of its $m$ marked
points (see Definition \ref{def: evaluation} below), and interpret counting of curves in terms
of intersections in homology, see Theorem \ref{thm: count}.
In order for the homology intersection theory to work, we shall need a certain balance of
dimensions, assured by the fact that for $m<n$ the maximal number of bounded edges is
exactly $2m-2$.
\end{rem}

Denote by $\mathcal{M}_{\Delta}(\mu)$ the set of rigid $m$-marked curves of a
combinatorial type $\mu$ and the $\Delta$-set $\Delta$.
We define coordinates and orientations on the space $\mathcal{M}_{\Delta}(\mu)$ as in
Section \ref{subsec: Moduli and orientation} above.
Namely, after an ordering of the set $E^0$ of internal edges of a graph $\G$ of the combinatorial
type $\mu$ and a choice of a root vertex $v_*$, the space $\mathcal{M}_{\Delta}(\mu)$ can be
identified with an (unbounded) open convex polyhedron $(\RR_+)^{|E^0|}\times\RR^2$ in a real
vector space $\RR^{|E^0|}\times\RR^2$.
An orientation of $\mathcal{M}_{\Delta}(\mu)$ is once again given by a sign-ordering of the set
$E^0$ (together with the standard orientation on the $\RR^2$-factor).
Define $\codim(\mu):= m-n+1 + \sum\limits_{v\in V}(\mathrm{val}(v)-3)$.
Then (since we work with trees) we have $|E^0|=2m-2- \codim(\mu)$, so the dimension
of $\mathcal{M}_{\Delta}(\mu)$ equals to
$$\dim \mathcal{M}_{\Delta}(\mu)=2m - \codim(\mu)\,.$$

We define the closure $\overline{\mathcal{M}}_{\Delta}(\mu)$ in a straightforward manner,
as the standard closure $(\RR_{\ge 0})^{|E^0|}\times\RR^2$ of the open polyhedron
$(\RR_+)^{|E^0|}\times\RR^2$ in $\RR^{|E^0|}\times\RR^2$.
Each relatively open cell in $\partial\big( (\RR_{\ge 0})^{|E^0|}\times\RR^2\big)$ corresponds to the
interior of the intersection of $(\RR_{\ge 0})^{|E^0|}\times\RR^2$ with a hyperplane $\{l_e=0\}$ for
some $e\in E^0$ and can be identified with $\mathcal{M}_{\Delta}(\mu_e)$, where $\mu_e$ is a
combinatorial type of a graph $\G/e$, obtained from $\G$ by a contraction of the (bounded) edge $e$.
It follows that $\partial \overline{\mathcal{M}}_{\Delta}(\mu)$ is the union
$\cup_I\mathcal{M}_{\Delta}(\mu_I)$ over non-empty subsets $I\subset E^0$ and $\mu_I$ is a
combinatorial type of a graph obtained from $\G$ by contractions of all edges in $I$.

Denote by $\mathcal{M}_{\Delta,n,m}$ the set of all rigid $m$-marked irreducible curves with the
$\Delta$-set $\Delta$, $|\Delta|=n$.
The space $\mathcal{M}_{\Delta,n,m}$ is stratified by combinatorial types of curves and we have
\begin{equation}\label{eqn: gluings}
\mathcal{M}_{\Delta,n,m} = \bigcup_\mu\mathcal{M}_{\Delta}(\mu) =
\Big(\bigsqcup_\mu\overline{\mathcal{M}}_{\Delta}(\mu)\Big)\big/\sim\,,
\end{equation}
where the union is over all combinatorial types $\mu$ of parameterized rigid $m$-marked curves with
the $\Delta$-set $\Delta$ and $C\sim C'$ if and only if $C=C'$.
The following proposition about the polyhedral structure of $\mathcal{M}_{\Delta,n,m}$ can be directly
transferred from \cite{GM1} including the proof:

\begin{prop}[Cf. {\cite[Proposition 3.12]{GM1}}]
For $m<n$ the moduli space $\mathcal{M}_{\Delta,n,m}$ has a structure of a
polyhedral complex of top dimension $2m$.
The corresponding (closed) faces of codimension $k$ are given by the sets
$\overline{\mathcal{M}}_{\Delta}(\mu)$ for all combinatorial types $\mu$ of $\codim(\mu)=k$.
\end{prop}

In the rest of the paper we shall consider the case of $m=n-1$ marked legs (so that
$\dim\mathcal{M}_{\Delta,n,n-1} =2m=2n-2$) that is relevant to enumerative problems
of Sections \ref{sec: homology}--\ref{sec: recursion}.
The case $m< n-1$ is related to a study of descendant Gromov-Witten-type invariants
and more complicated counting problems, so we postpone it to a forthcoming paper.

Let us study the structure of $\mathcal{M}_{\Delta,n,n-1}$ in the crucial cases $n=2,3$:
\begin{ex}
\label{ex: tripod moduli}
For $m=n-1=1$ and a $\Delta$-set $\Delta=\{\xi,-\xi\}$ there is only one combinatorial
type $\mu$ of rigid curves, shown in Figure \ref{fig: rigid curves}a, so that
$\mathcal{M}_{\Delta,2,1}\cong\RR^2$ (with the coordinate being the position of the only
marked vertex).

For $m=n-1=2$ and a $\Delta$-set $\Delta=\{\xi_1,\xi_2,\xi_3\}$ there are 6 top-dimensional
faces, corresponding to combinatorial types $\mu_{ij}$, $1\le i\ne j\le 3$ of curves, with $\mu_{ij}$
being the type corresponding to $z_1$ incident to $e_i$ and $z_2$ incident to $e_j$, see Figures
\ref{fig: rigid curves}b, \ref{fig: tripod}.
All six closed faces $\overline{\mathcal{M}}_{\Delta}(\mu_{ij})\cong (\RR_{\ge 0})^2\times\RR^2$
are glued along their boundary strata as shown in Figure \ref{fig: tripod}, resulting in the polyhedral
structure on $\mathcal{M}_{\Delta,3,2}\cong\RR^2\times\RR^2$.

\begin{figure}[htb]
%\vspace{1in}
\includegraphics[height=1.6in]{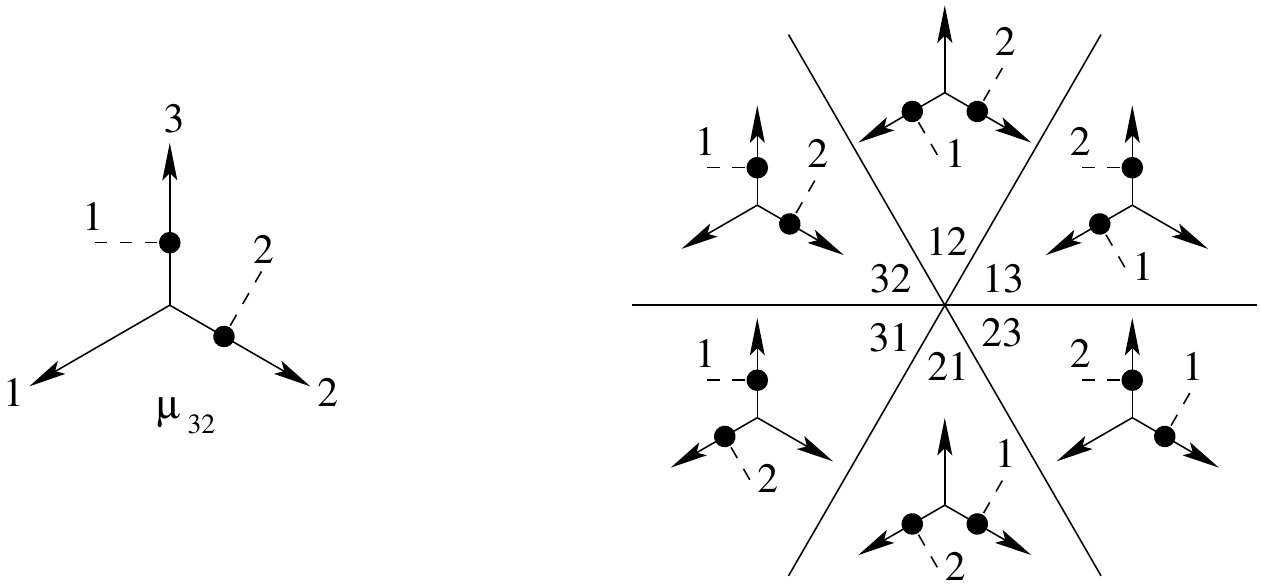}
\caption{Moduli of rigid curves for $m=n-1=2$.
\label{fig: tripod}}
\end{figure}
\end{ex}

\subsection{Top-dimensional faces of $\mathcal{M}_{\Delta,n,n-1}$ and their orientations}
\label{subsec: faces}
For $m=n - 1$ there is an additional structure on top-dimensional faces of $\mathcal{M}_{\Delta,n,m}$.
Indeed, let $\G$ be a rigid tree of a combinatorial type $\mu$ with $\codim(\mu) = 0$.
But for $m=n-1$ we have $\codim(\mu)=\sum\limits_{v\in V}(\mathrm{val}(v)-3)$,
so all (in particular, marked) vertices of such a tree $\G$ are trivalent.
Thus $\G\minus E^m$ consists of $m+1=n$ connected components and the rigidity condition
of Definition \ref{def: rigid curves} implies that every connected component of $\G\minus E^m$
contains exactly one leg.
This important observation leads, in particular, to an alternative description of orientations.

\begin{defn}
A {\em Lie-orientation} of a trivalent tree with $n$ unmarked and $m$ marked legs is a choice
of a cyclic ordering of half-edges in each of the $n-2$ unmarked vertices of the tree, up to negating
any two such cyclic orderings.
\end{defn}
We claim that a choice of a Lie-orientation of a rigid trivalent tree $\G$ of a combinatorial type
$\mu$ with $\codim(\mu) = 0$ is equivalent to an orientation of $\mathcal{M}_{\Delta}(\mu)$.
Indeed, fix a cyclic order of half-edges at every unmarked vertex.
As observed above, every connected component of $\G\minus E^m$ has exactly one leg;
direct all edges in each such connected component towards its unique leg, see Figure
\ref{fig: orientation}a.
We shall call such a choice of directions on edges the {\em outer flow}.

\begin{figure}[htb]
%\vspace{1in}
\includegraphics[width=5in]{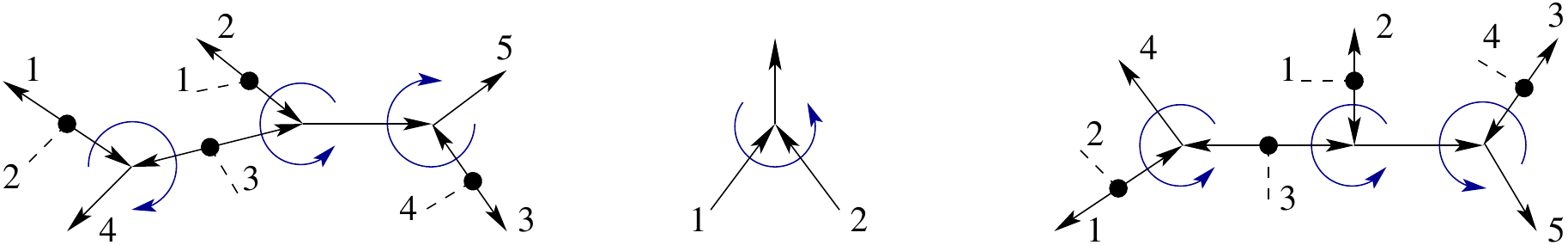}\\
a\hspace{1.5in}b\hspace{1.5in}c
\caption{Outer flow and cyclic orientations.
\label{fig: orientation}}
\end{figure}

In each unmarked vertex $v$ there are exactly two incoming and one outgoing half-edges.
The cyclic order determines the order $e_1(v)\wedge e_2(v)$ of the two incoming half-edges
at $v$, see Figure \ref{fig: orientation}b.
Note that since these half-edges are incoming, they belong to bounded edges only and vice-versa:
for every bounded edge one of its half-edges is incoming for some unmarked vertex.
Order unmarked vertices $v_i$, $1\le i\le n-2$ arbitrarily.
A sign-ordering
$$\left(e_1(v_1)\wedge e_2(v_1)\right)\wedge \left(e_1(v_2)\wedge e_2(v_2)\right)
\wedge\dots\wedge \left(e_1(v_{n-2})\wedge e_2(v_{n-2})\right)$$
determines an orientation on $\mathcal{M}_{\Delta}(\mu)$.
Reorderings of $n-2$ unmarked vertices lead to even permutations of $2(n-2)=2m-2$ bounded
edges and thus preserve this orientation.
We thus obtain

\begin{prop}
Oriented top-dimensional faces of $\mathcal{M}_{\Delta,n,n-1}$ are in one-to-one correspondence
with Lie-oriented rigid trivalent trees with $n$ unmarked and $n-1$ marked legs.
\end{prop}

We shall call such Lie-oriented rigid trivalent trees {\em $(n,n-1)$-trees} for short.
In figures we shall always draw them using the outer flow directions on edges and the
counterclockwise cyclic ordering of half-edges in all unmarked vertices (unless indicated otherwise).
See Figure \ref{fig: orientation}c.

\subsection{Multiplicities, orientations and evaluation map}
\label{subsec: evaluation}
Fix an orientation $or_\mu$ on a top-dimensional face $\mathcal{M}_{\Delta}(\mu)$ of
$\mathcal{M}_{\Delta,n,n-1}$.
\begin{defn}
Let $C=[(\G,\ell,h,m)]$ be a curve in $\mathcal{M}_{\Delta}(\mu)$.
Fix  an ordering $e_1(v)\wedge e_2(v)$ of the two incoming half-edges
at each unmarked vertex $v$, consistent with the orientation $or_\mu$.
Define the {\em multiplicity $\mult(C, or_\mu)$ of $C$} as
\begin{equation}\label{eq: mult}
\mult(C, or_\mu):=\prod_v \det\left(\xi_v(e_1(v)),\xi_v(e_2(v))\right)\,,
\end{equation}
where the product is over all $n-2$ unmarked vertices of $C$.
\end{defn}

By definition, $\mult(C, or_\mu)$ depends only on the sign-ordering of bounded edges,
i.e., only on the orientation $or_\mu$ and $\mult(C, -or_\mu)=-\mult(C, or_\mu)$.
Recall that $C$ is non-degenerate, if for every unmarked vertex $v$ of $\G$ vectors
$\xi_v(e)$ span $\RR^2$.
Non-degeneracy of $C$ can be restated as $\mult(C, or_\mu)\ne 0$. For a non-degenerate
curve $C$ we call the orientation $or_\mu$ {\em positive}, if  $\mult(C, or_\mu)> 0$.
We shall alternatively call it the {\em blackboard orientation}, since the choice of the blackboard
order of incoming slope vectors (and the induced order of the corresponding half-edges) at each
unmarked vertex of $C$ defines the positive orientation on $\mathcal{M}_{\Delta}(\mu)$.

Recall that slope vectors on edges of $C$ depend only on $\mu$ (see discussion in Section
\ref{subsec: markedcurves} and Corollary \ref{cor: Uniqueness of extension}).
Thus we may denote $\mult(\mu, or_\mu)=\mult(C, or_\mu)$.

\begin{defn}
\label{def: evaluation}
We define the evaluation map $\ev$ by
$$\ev: \mathcal{M}_{\Delta,n,m}\to(\RR^2)^{m},\hspace{0.5cm} \ev([(\G, \ell, h, m)]) =
(h({z_1}),\dots, h({z_m})).$$
\end{defn}
Recall that $\dim\mathcal{M}_{\Delta,n,m} =2m$, so $\ev$ is a map between spaces of the
same dimension.
Denote by $\ev\restrict{\mu}$ the restriction of $\ev$ to $\overline{\mathcal{M}}_{\Delta}(\mu)$.

\begin{ex}
\label{ex: rigid}
Consider a curve $C$ with $m=n-1=1$ of Figure \ref{fig: rigid curves}a.
The unique marked vertex is also a root $v_*$, thus $\ev:\RR^2\to\RR^2$ is the
identity map and $\mult(C,+1)=1$, where $+1$ stands for the positive orientation of
$\mathcal{M}_{\Delta,2,1}\cong\RR^2$.

Consider a Lie-oriented curve $C$ of Figure \ref{fig: rigid curves}b with $m=n-1=2$ and a
$\Delta$-set $\Delta=\{\xi_1,\xi_2,\xi_3\}$.
We have $\mathcal{M}_{\Delta,3,2}\cong\RR^2\times\RR^2$, glued from six top-dimensional
closed faces $\overline{\mathcal{M}}_{\Delta}(\mu_{ij})$, see Example \ref{ex: tripod moduli}
and Figure \ref{fig: tripod}.
Choose the orientation $or_{\mu_{ij}}$ on $\overline{\mathcal{M}}_{\Delta}(\mu_{ij})$
given by the counterclockwise cyclic order of edges in the only unmarked vertex $v$ for each
combinatorial type $\mu_{ij}$ shown in Figure \ref{fig: tripod}. These orientations induce
opposite orientations on common boundaries of $\overline{\mathcal{M}}_{\Delta}(\mu_{ij})$
and thus define a global orientation on $\mathcal{M}_{\Delta,3,2}\cong\RR^2\times\RR^2$.

Define $\eps_{ij}=+1$ if $(ijk)$ is an even permutation of $(123)$ and $\eps_{ij}=-1$ if
$(ijk)$ is odd, see Figure \ref{fig: tripod}.
Let $\xi_i=(x_i,y_i)$ and  $\xi_j=(x_j,y_j)$ be the slope vectors along legs $e_i$ and $e_j$
(and thus, by balancing, also on bounded edges connecting $v$ with $z_1$ and $z_2$).
The multiplicity of $C$ is $\mult(C, or_{\mu_{ij}})=\det(\xi_i,\xi_j)$ if $\eps_{ij}=+1$ and
$\mult(C, or_{\mu_{ij}})=\det(\xi_j,\xi_i)$ if $\eps_{ij}=-1$. Thus by the balancing
condition $\mult(C, or_{\mu_{ij}})=\det(\xi_1,\xi_2)=\det(\xi_2,\xi_3)=\det(\xi_3,\xi_1)$
is constant on the whole moduli space $\mathcal{M}_{\Delta,3,2}$.

Denote by $l_1$ and $l_2$ lengths of bounded edges connecting $v$ with marked vertices
$z_1$ and $z_2$, respectively.
The Lie-orientation $or_{\mu_{ij}}$ of $C$ determines the order of the two incoming
half-edges of $C$ in $v$ and hence defines the orientation $\eps_{ij}(l_1\wedge l_2)$ of
$\RR^{|E^0|}=\RR^2$, see Figure \ref{fig: tripod}.
Choose $v_*=z_2$ as the root vertex.
Then $\ev:\mathcal{M}_{\Delta,3,2}\cong\RR^2\times\RR^2\to(\RR^2)^2$ is given
by a PL-map, glued from linear maps $\ev\restrict{\mu_{ij}}$ on each of the six faces
$\overline{\mathcal{M}}_{\Delta}(\mu_{ij})$ of $\mathcal{M}_{\Delta,3,2}$.
These maps are given by
$$\ev\restrict{\mu_{ij}}:(l_1,l_2,h(v_*))\mapsto(\ev_{ij}(l_1,l_2)+h(v_*),h(v_*))\,,$$
where $\ev_{ij}(l_1,l_2)=l_1\xi_i-l_2\xi_j$.
In the positive coordinates $(l_1,l_2,h(v_*))$ for $l_1\wedge l_2$-orientation and
$(l_2,l_1,h(v_*))$ for $l_2\wedge l_1$-orientation the matrix of $\ev\restrict{\mu_{ij}}$ is
$\left(\begin{array}{c|c}
          A_{ij}\sigma_{ij} & I_2 \\
          \hline
          0 & I_2 \\
        \end{array}\right)$, where $I_2$ is the $2\times 2$ identity matrix, $A_{ij}$ is the matrix
$\left(\begin{array}{cc}
          x_i & -x_j \\
          y_i & -y_j \\
        \end{array}\right)$
and $\sigma_{ij}=I_2$ if $\eps_{ij}=+1$ and
$\sigma_{ij}=\left(\begin{array}{cc}
          0 & 1 \\
          1 & 0 \\
        \end{array}\right)$
if $\eps_{ij}=-1$.
The determinant of $\ev\restrict{\mu_{ij}}$ is
$$\eps_{ij}(-x_iy_j+x_jy_i)=-\eps_{ij}\det(\xi_i,\xi_j)=-\mult(C, or_{\mu_{ij}})\,.$$

If vectors of $\Delta$ do not span $\RR^2$, then $\text{rank}(A_{ij})=1$,
$\mult(C, or_{\mu_{ij}})=0$ and the image of $\ev$ is 3-dimensional.
If vectors of $\Delta$ span $\RR^2$, then $\text{rank}(A_{ij})=2$, $\mult(C, or_{\mu_{ij}})\ne 0$
and the image $\bigcup_{i,j}\ev_{ij}(\overline{\mathcal{M}}_{\Delta}(\mu_{ij}))$ is $\RR^2$ with
the polyhedral structure induced from that on $\mathcal{M}_{\Delta,3,2}$, see Figure \ref{fig: compass}.
The resulting PL-map $\ev:\RR^4\to\RR^4$ is a bijection of degree
$-\sign\left(\mult(C, or_{\mu_{ij}})\right)=-\sign\left(\det(\xi_1,\xi_2)\right)$.
In particular, if one chooses the blackboard orientation on each
$\overline{\mathcal{M}}_{\Delta}(\mu_{ij})$ then $\ev:\RR^4\to\RR^4$ is a bijection of degree $-1$.
\begin{figure}[htb]
%\vspace{1in}
\includegraphics[height=1.2in]{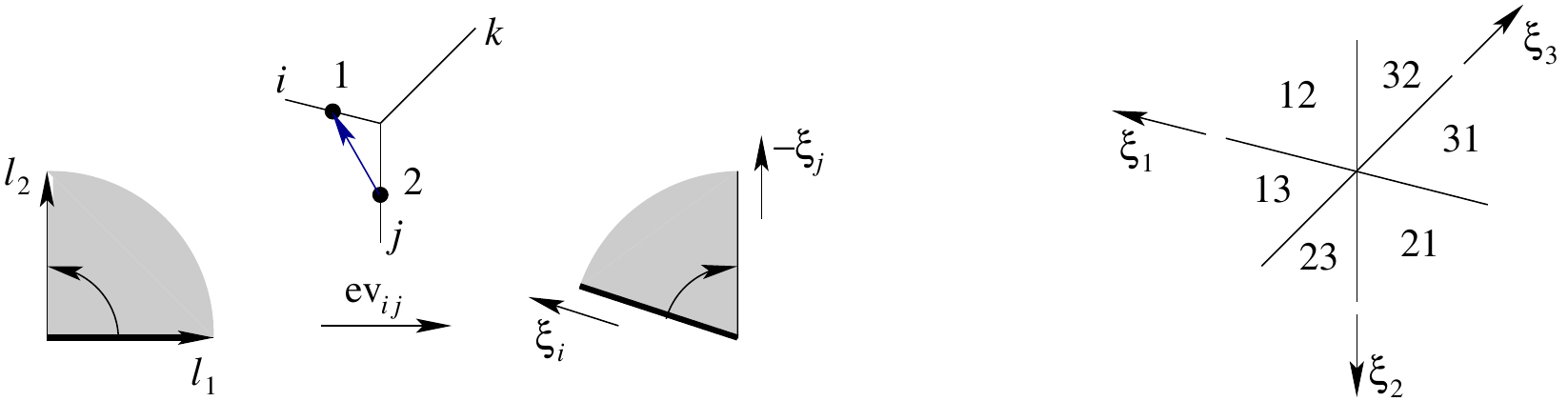}
\caption{PL-structure of $\ev$.
\label{fig: compass}}
\end{figure}
\end{ex}

\begin{prop}
\label{prop: Structure of evaluation map}
Under the identification $\overline{\mathcal{M}}_{\Delta}(\mu)\cong(\RR_{\ge 0})^{|E^0|}\times\RR^2$
given by an ordering of $E^0$ and a choice of a root vertex (and the standard coordinates on $(\RR^2)^m$)
the evaluation map $\ev$ is linear on each face $\overline{\mathcal{M}}_{\Delta}(\mu)$.
On a top-dimensional face $\mathcal{M}_{\Delta}(\mu)$ of $\mathcal{M}_{\Delta,n,n-1}$ equipped
with an orientation $or_\mu$ the determinant of $\ev\restrict{\mu}$ in the coordinates above equals
$(-1)^n\mult(\mu, or_\mu)$.
\end{prop}

\begin{proof}
For any marked vertex $z_i$, $i=1,\dots,n-1$ there exists a unique 1-chain $C_i=\sum_{e\in C_i} e$
of oriented bounded edges of $\G$ such that $\partial C_i=z_i-v_*$.
Since the components of $\ev$ are $h(z_i)=\sum_{e\in C_i} l_e\cdot\xi(e)+h(v_*)$, every $h(z_i)$
is given by a linear combination of certain $l_e$'s and $h(v_*)$, so $\ev$ is linear on each face
$\overline{\mathcal{M}}_{\Delta}(\mu)$.

To compute the determinant of $\ev$ on a top-dimensional face $\mathcal{M}_{\Delta}(\mu)$
we closely follow the proof of \cite[Proposition 3.8]{GM1}, modifying it appropriately for the
pseudotropical case. We proceed by induction on the number $|E^0|=2m-2$ of bounded edges.
For $m=1,2$ the statement follows from calculations in Example \ref{ex: rigid}. For $m>2$, let
$C=[(\G, \ell, h, m)]$ be a rigid curve of type $\mu$ (with the Lie-orientation $or_\mu$ and the
outer flow directions of edges). We shall use the edge-cut procedure shown in Figure \ref{fig: edge_cut}
and described below to split $C$ into two rigid curves $C^\pm$ with smaller numbers of bounded edges.

\begin{figure}[htb]
%\vspace{1in}
\includegraphics[width=5in]{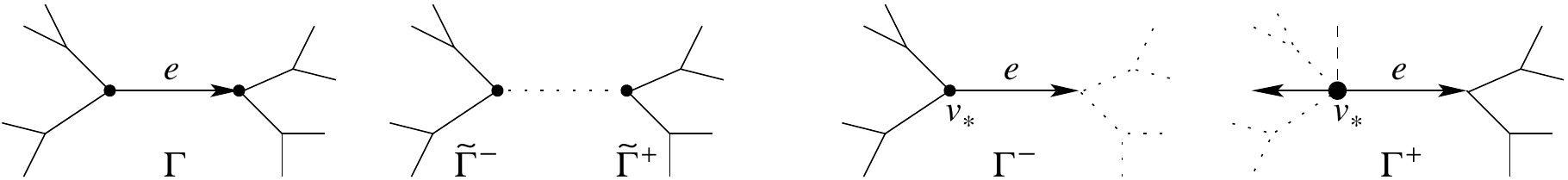}\\
a\hspace{2.6in}b
\caption{Edge-cut of a rigid curve.
\label{fig: edge_cut}}
\end{figure}

Choose a bounded edge $e$ of $\G$, such that each connected component of $\G$
after the removal of an interior of the edge $e$ contains at least one bounded edge.
Let $\widetilde\G^\pm$ be the connected component containing $e^\pm$,
see Figure \ref{fig: edge_cut}a.

If $\xi(e)=0$, then $\ev\restrict{\mu}$ is not injective (since all curves which differ
by the length $l_e$ of $e$ have the same image), so $\det(\ev\restrict{\mu})=0$.
On the other hand, in this case slope vectors incident to $e^+$ do not span $\RR^2$
so $\mult(C, or_\mu)=0$ and the proposition follows.
Therefore we may assume that $\xi(e)\ne 0$.

Equip $\widetilde\G^\pm$ with directions, lengths and slope vectors along all edges as in $\G$.
Denote by $\G^-$ the graph obtained from $\widetilde\G^-$ by an addition of a new unmarked
leg incident to $e^-$ with the slope vector $\xi(e)$.
In a similar way, denote by $\G^+$ the graph obtained from $\widetilde\G^+\cup \{e\}$  by
an addition of one new marked and one unmarked leg incident to the vertex $e^-$, with slope
vectors $0$ and $-\xi(e)$ respectively, see Figure \ref{fig: edge_cut}b.
Denote by $m^\pm$ the number of marked legs of $\widetilde\G^\pm$ (so $m^-+m^+=m+1$).
Order the legs of $\G^\pm$ following their relative order in $\G$ (with the new marked leg on
$\G^+$ being the last marked leg).
The outer flow and the Lie-orientation of $\G$ induce those on $\G^\pm$ and make them into a
pair of rigid Lie-oriented curves $C^\pm$ with $m^\pm$ marked and $m^\pm+1$ unmarked legs.

Note that the determinant of $\ev$ does not depend on an order of legs, order of bounded
edges and a choice of the root vertex, but only on the orientation of the corresponding curve.
Thus we may choose orders of legs and bounded edges of $\G$ so that edges of $\G^-$
precede those of $\G^+$ and use the same ordering (up to a shift) on $\G^\pm$.
Also, we  may choose $v_*=e^-$ as the root vertex on all three curves $\G$, $\G^\pm$.
With these choices the matrix of $\ev$ has the block form
$ \left(        \begin{array}{c|c}
          A^- & 0 \\
          \hline
          0 & A^+ \\
        \end{array}      \right)$,
while matrices of $\ev$ on $\G^\pm$ are $A^-$ and $\left(        \begin{array}{c|c}
          A^+ & 0 \\
          \hline
          0 & I_2 \\
        \end{array}      \right)$, respectively. The induction hypothesis implies the proposition.

\end{proof}

\begin{cor}
\label{cor: orientation}
If $\mult(\mu,or_\mu)\ne 0$, the map $\ev$ is injective on $\mathcal{M}_{\Delta}(\mu)$
and the blackboard orientation on $\mathcal{M}_{\Delta}(\mu)$ maps to $(-1)^n$
times the standard orientation of $(\RR^2)^{n-1}$.
\end{cor}

\section{Homology of moduli spaces}
\label{sec: homology}
In this section we study the top homology of the compactified moduli space
$\widehat{\mathcal{M}}_{\Delta,n,n-1}$ and identify it with a certain quotient
of the space of $(n,n-1)$-trees.

Note that cells of the polyhedral complex $\mathcal{M}_{\Delta,n,n-1}$ are not compact;
we shall compactify this complex as follows.
Since $\mathcal{M}_{\Delta,n,n-1}$ is locally compact and Hausdorff, it admits a one-point compactification
$\widehat{\mathcal{M}}_{\Delta,n,n-1} = \mathcal{M}_{\Delta,n,n-1}\cup\{C_\infty\}$.
Here the ``infinite'' point $C_\infty$ corresponds to all sequences of rigid $(n-1)$-marked
curves such that some of their lengths $l_e$, $e\in E^0$ tend to $+\infty$.

Formula \eqref{eqn: gluings} implies that the top homology
$H_{2n-2}(\widehat{\mathcal{M}}_{\Delta,n,n-1};\QQ)$
is isomorphic to a $\QQ$-vector space spanned by oriented compactified closed cells
$\widehat{\mathcal{M}}_{\Delta}(\mu)=\overline{\mathcal{M}}_{\Delta}(\mu)\cup\{C_\infty\}$
for all combinatorial types $\mu$ with $\codim(\mu) = 0$ modulo gluing relations, dictated by 
gluings in \eqref{eqn: gluings} along faces of codimension one.
Note that such faces correspond to a contraction of some bounded edge in a graph of type $\mu$.
Since both cells and gluing relations are encoded by combinatorial types of graphs, the space
$H_{2n-2}(\widehat{\mathcal{M}}_{\Delta,n,n-1};\QQ)$ is isomorphic to the $\QQ$-vector space
spanned by combinatorial types $\mu$ of trivalent trees appearing in \eqref{eqn: gluings} modulo
relations arising from identifications of common combinatorial types of graphs $\mu_e$, obtained
by contractions of a bounded edge $e$.
This is a certain genus-$0$ version of Kontsevich's graph homology, see \cite{K}.

\subsection{Gluings of top-dimensional faces}
\label{subsec: gluings of faces}
Recall once again that for an $(n,n-1)$-tree $\G$ of type $\mu$ every connected component of
$\G\minus E^m$ has exactly one leg.
Thus there are no bounded edges in $\G$ with both ends in marked vertices; the only types of
bounded edges of $\G$ are those incident either to two unmarked vertices or to one marked and
one unmarked vertex. Let us call them edges of type I and type II, respectively.
It follows that there are only two kinds of combinatorial types $\mu_e$ on the boundary of
$\overline{\mathcal{M}}_{\Delta}(\mu)$ with $\codim(\mu_e) = 1$ corresponding to a
contraction of a bounded edge $e$ of type I or II, respectively.
We shall denote the corresponding parts of the boundary of $\overline{\mathcal{M}}_{\Delta}(\mu)$
by $\partial_{\I}\overline{\mathcal{M}}_{\Delta}(\mu)$ and
$\partial_{\II}\overline{\mathcal{M}}_{\Delta}(\mu)$, respectively.
Let us describe combinatorial types $\mu$ that degenerate to a fixed combinatorial type $\nu$ in
$\partial_{\I}\overline{\mathcal{M}}_{\Delta}(\mu)$ or
$\partial_{\II}\overline{\mathcal{M}}_{\Delta}(\mu)$.
A graph of such a combinatorial type $\nu$ is a rigid tree with all trivalent vertices, except for
one four-valent vertex.
If we forget the rigidity condition, then there are three combinatorial types of trivalent trees degenerating
to $\nu$, obtained from $\nu$ by splitting the four-valent vertex into two adjacent trivalent vertices.
Orienting all edges as above, we see that for $\nu\in\partial_{\I}\overline{\mathcal{M}}_{\Delta}(\mu)$
all three types are rigid, see Figure \ref{fig: boundary I}.

\begin{figure}[htb]
%\vspace{1in}
\includegraphics[width=4.5in]{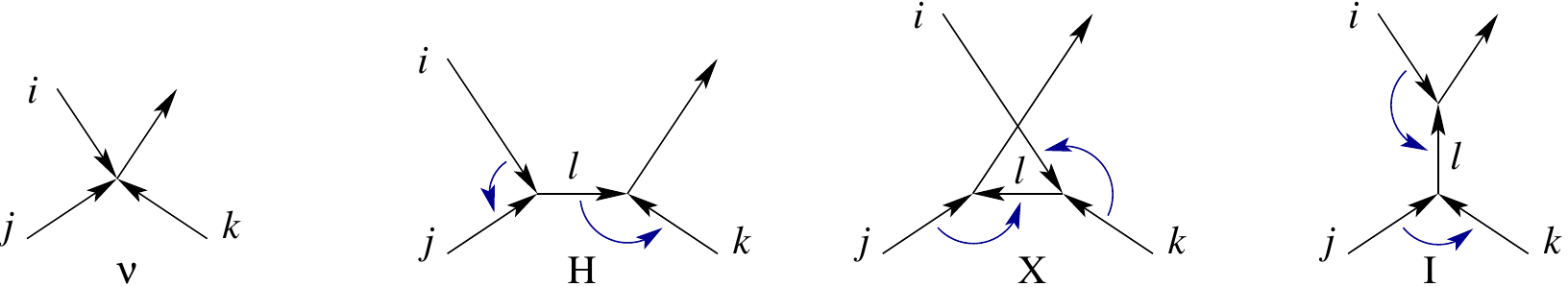}
\caption{Boundary of type I.
\label{fig: boundary I}}
\end{figure}

For $\nu\in\partial_{\II}\overline{\mathcal{M}}_{\Delta}(\mu)$, however, one of these trees
(the marked $I$-graph) does not satisfy the rigidity condition of Definition \ref{def:  rigid curves},
see Figure \ref{fig: boundary II}.

\begin{figure}[htb]
%\vspace{1in}
\includegraphics[width=5in]{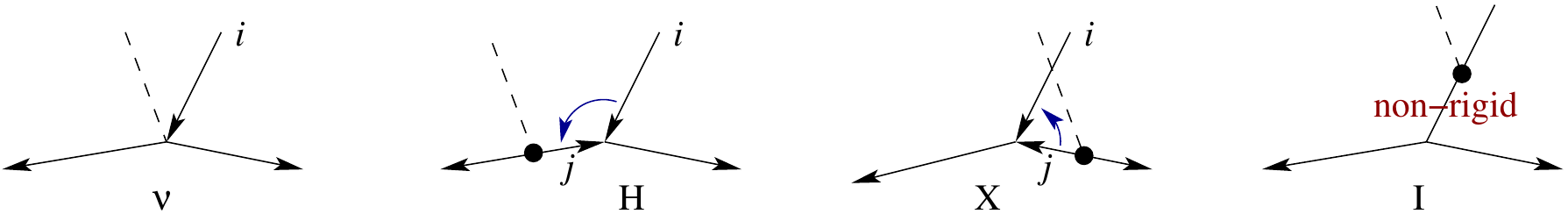}
\caption{Boundary of type II.
\label{fig: boundary II}}
\end{figure}

Thus in the space $\widehat{\mathcal{M}}_{\Delta,n,n-1}$ top-dimensional faces are glued together
along their common boundaries in triples in the case of type I and in pairs in the case of type II.
See Figure \ref{fig: gluing}.

\begin{figure}[htb]
%\vspace{1in}
\includegraphics[width=5in]{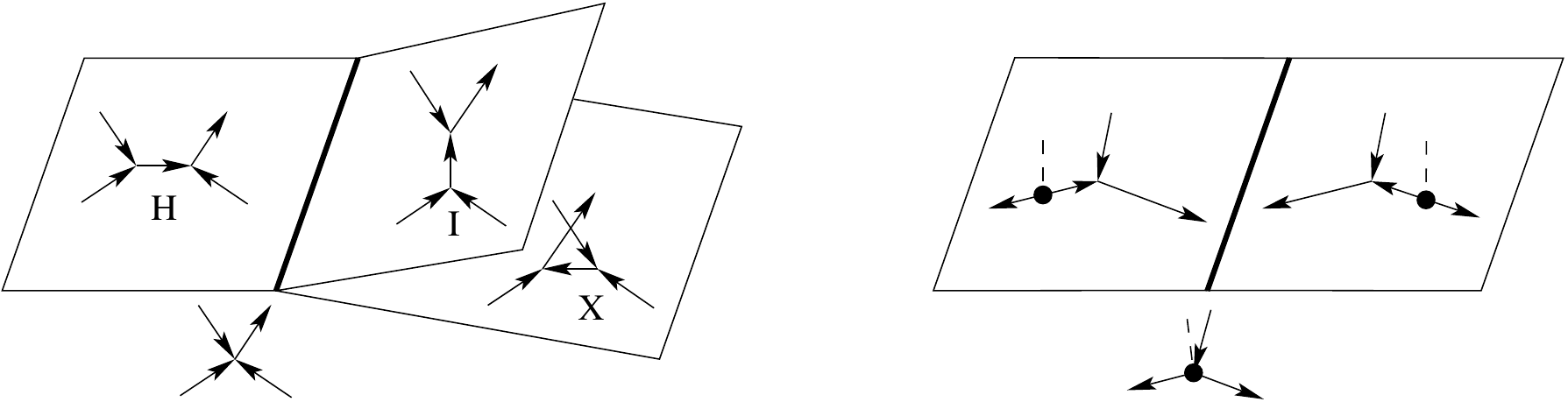}
\caption{Gluing faces along common boundaries.
\label{fig: gluing}}
\end{figure}

To understand the gluing relations we should determine which orientations of top-dimensional
faces induce the same orientation on their common boundary of codimension one.
A comparison of Lie-orientations of $IHX$-graphs with orderings of half-edges shown in
Figure \ref{fig: boundary I} gives
$\left(e_i\wedge e_j\right)\wedge \left(e_l\wedge e_k\right)$ for the $H$-graph,
$\left(e_j\wedge e_l\right)\wedge \left(e_k\wedge e_i\right)$ for the $X$-graph, and
$\left(e_j\wedge e_k\right)\wedge \left(e_i\wedge e_l\right)$ for the $I$-graph, so
these Lie-orientations of $X$ and $I$-graphs induce the same orientation on the
common boundary face, while the orientation induced by the $H$-graph is the
opposite.
Similarly, a comparison of Lie-orientations with orderings of half-edges shown in
Figure \ref{fig: boundary II} gives $e_i\wedge e_j$ for the marked $H$-graph
and $e_j\wedge e_i$  for the marked $X$-graph, so these Lie-oriented graphs
induce opposite orientations on the common boundary face.

Orientations of top-dimensional faces and combinatorics of their gluings along
$\partial_{\I}\overline{\mathcal{M}}_{\Delta}(\mu)$ and
$\partial_{\II}\overline{\mathcal{M}}_{\Delta}(\mu)$ motivates the following definition:

\begin{defn}
Let $J^{rig}_{n,n-1}$ be the $\QQ$-vector space spanned by isomorphism classes of
$(n,n-1)$-trees modulo the following antisymmetry (AS), Jacobi (IHX) and marked point (MP) relations:
\begin{equation}\label{eq: AS}
\raisebox{-0.5\height}{\includegraphics[height=0.5in]{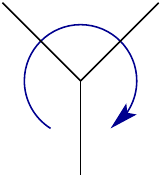}}\  =\ - \
\raisebox{-0.5\height}{\includegraphics[height=0.5in]{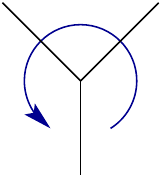}}
\tag{AS}
\end{equation}
%\vspace{0.1cm}
\begin{equation}\label{eq: IHX}
\raisebox{-0.5\height}{\includegraphics[height=0.45in]{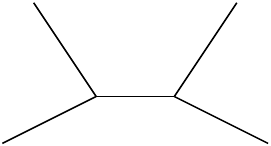}}\  - \
\raisebox{-0.5\height}{\includegraphics[height=0.45in]{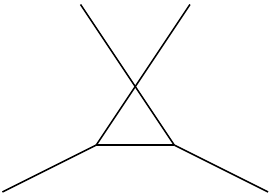}}\  - \
\raisebox{-0.5\height}{\includegraphics[height=0.45in]{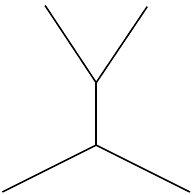}}\  = 0
\tag{IHX}
\end{equation}
%\vspace{0.1cm}
\begin{equation}\label{eq: MP}
\raisebox{-0.5\height}{\includegraphics[height=0.5in]{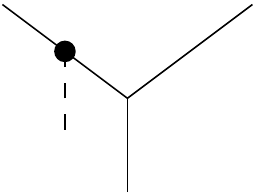}}\  -\
\raisebox{-0.5\height}{\includegraphics[height=0.5in]{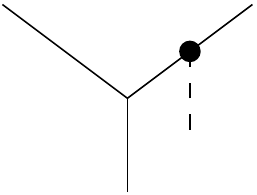}}\ = 0
\tag{MP}
\end{equation}
\end{defn}
Here, in each relation, we assume that the graphs (including cyclic orders at each unmarked
vertex) are the same outside the indicated fragments and all unmarked vertices in figures
are oriented counterclockwise unless shown otherwise.

Summarizing analysis in Section \ref{sec: homology} we conclude
\begin{thm}
\label{thm: Jacobi homology}
$H_{2n-2}(\widehat{\mathcal{M}}_{\Delta,n,n-1};\QQ)$ is isomorphic to the space $J^{rig}_{n,n-1}$
of $(n,n-1)$-trees.
\end{thm}

Note that while the moduli space itself explicitly depends on the $\Delta$-set, the structure of
its top homology $H_{2n-2}(\widehat{\mathcal{M}}_{\Delta,n,n-1};\QQ)$ depends only on $|\Delta|=n$.
Unfortunately, the number of generators of the vector space $J^{rig}_{n,n-1}$ rapidly grows with
$n$: for $n=3$ up to a choice of a Lie-orientation there are six $(3,2)$-trees (which correspond to six
top-dimensional faces $\mu_{ij}$ of $\mathcal{M}_{\Delta,3,2}$, see Example \ref{ex: tripod moduli}),
while for $n=4$ there are already 48 corresponding $(4,3)$-trees.
While these are convenient for understanding the polyhedral structure of ${\mathcal{M}}_{\Delta,n,n-1}$,
top homology $H_{2n-2}(\widehat{\mathcal{M}}_{\Delta,n,n-1};\QQ)$ can be better understood in
terms of unmarked trees.
Indeed, relations (MP) allow one to significantly reduce the number of generators, identifying
rigid trees which differ by positions of marked points. This leads to an alternative definition of
the vector space $J^{rig}_{n,n-1}$.

Namely, define a forgetful ``marking-removal'' map $f$ as follows: remove all open marked legs
and ``smooth out'' the resulting marked 2-valent vertices of an $(n,n-1)$-tree $T$ to obtain a
trivalent tree $f(T)$ with $n$ ordered unmarked legs.
A Lie-orientation of $f(T)$ is induced from $T$.
\begin{defn}
\label{def: Jacobi}
The {\em Jacobi space} $J_n$ is the $\QQ$-vector space spanned by isomorphism
classes of Lie-oriented trivalent trees with $n$ (unmarked) legs modulo the antisymmetry
(AS) and Jacobi (IHX) relations.
\end{defn}

\begin{figure}[htb]
%\vspace{1in}
\includegraphics[width=5in]{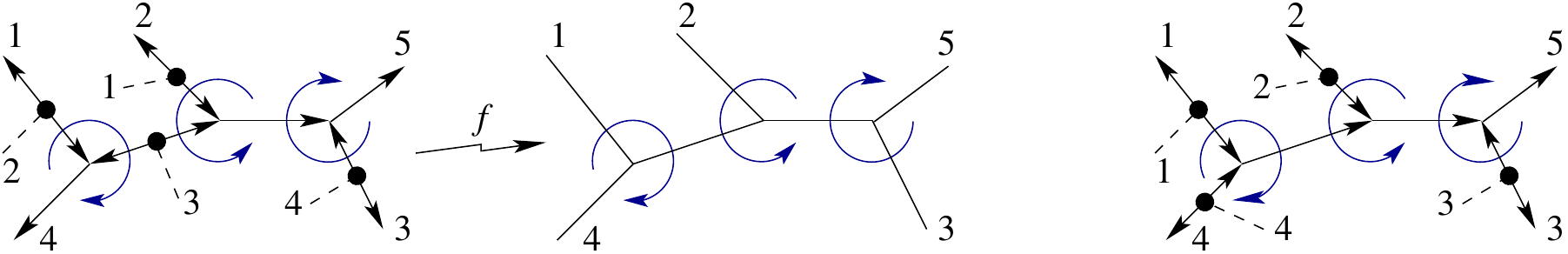}\\
\hspace{0.9in}a\hspace{2.7in}b
\caption{Marking-removal and marked canonical representatives.
\label{fig: forgetful map}}
\end{figure}

\begin{lem}
\label{lem: forgetful map}
The quotient map $f$ defines an isomorphism $f: J^{rig}_{n,n-1}\to J_n$ of vector spaces.
\end{lem}
\begin{proof}
In each equivalence class $T$ of Lie-oriented rigid marked trees we choose as a canonical
representative a Lie-oriented tree $T_*$ with each marked vertex $z_i$, $i=1,\dots,n-1$
adjacent to the corresponding unmarked leg $e_i$ (except for the last unmarked leg $e_n$).
See Figure \ref{fig: forgetful map}b.
It is easy to check that relations (MP) allow one to identify any other representative of this class
with $T_*$. Therefore out of all (AS) and (IHX) relations we may leave only relations between
canonical representatives. The identification of unmarked trees in $J_n$ with canonical
representatives in $J^{rig}_{n,n-1}$ gives the desired isomorphism.
\end{proof}

\begin{cor}
\label{cor: Jacobi homology}
$H_{2n-2}(\widehat{\mathcal{M}}_{\Delta,n,n-1};\QQ)$ is isomorphic to the Jacobi space $J_n$
of Lie-oriented trivalent trees.
\end{cor}
\begin{ex}
\label{ex: Jacobi spaces}
For $n=3$ there is only one type of trivalent trees, so the space $J_3$ is one-dimensional;
and indeed $H_4(\widehat{\mathcal{M}}_{\Delta,3,2};\QQ)=H_4(S^4;\QQ)=\QQ$, see
Example \ref{ex: tripod moduli}.

For $n=4$ the space $J_4$ is generated by three types $I$, $H$, $X$ of trivalent trees
with four ordered legs modulo the (IHX) relation. The combinatorics of this space is
described by the tripod graph, see Figure \ref{fig: Jacobi spaces}a.
Thus $H_6(\widehat{\mathcal{M}}_{\Delta,4,3};\QQ)\cong J_4$ is two-dimensional.
Two generating cycles are e.g. $Z_{HX}$ and $Z_{HI}$ with $Z_{HX}(H)=Z_{HX}(X)=1$,
$Z_{HX}(I)=0$ and $Z_{HI}(H)=Z_{HI}(I)=1$, $Z_{HI}(X)=0$.

For $n=5$ the space $J_5$ is generated by 15 types of trivalent trees with five ordered legs,
modulo 10 (IHX) relations. The combinatorics of this space is described by the Petersen graph,
see Figure \ref{fig: Jacobi spaces}b.
Thus $H_8(\widehat{\mathcal{M}}_{\Delta,5,4};\QQ)\cong J_5$ is 6-dimensional.
Some generating cycles related to an enumeration of certain curves will be constructed in
Example \ref{ex: caterpillar cycles}.
\end{ex}

\begin{figure}[htb]
%\vspace{1in}
\includegraphics[width=5in]{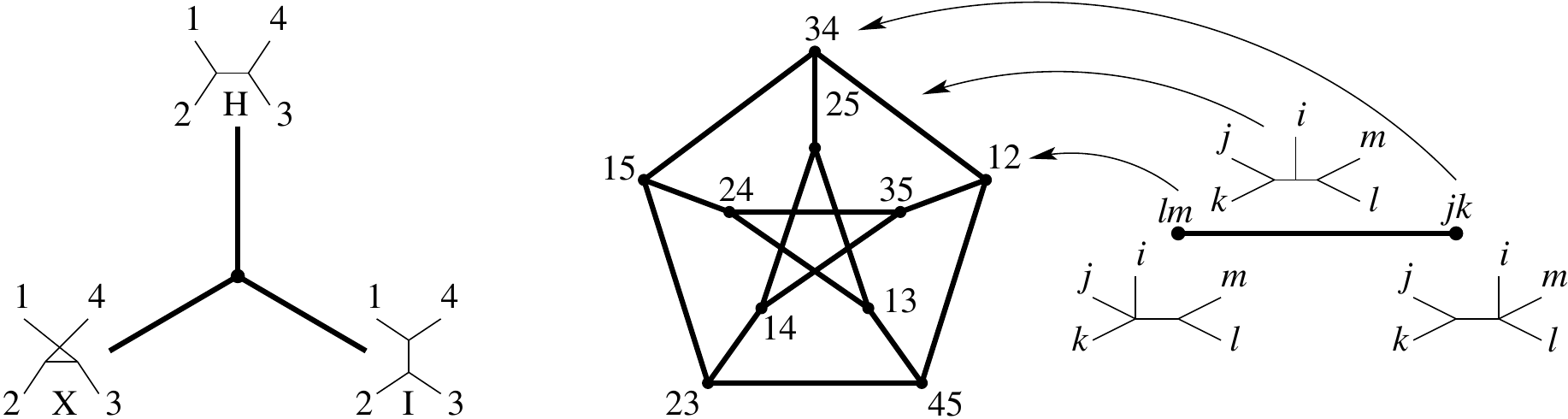}\\
a\hspace{2.7in}b\hspace{0.8in}
\caption{Combinatorics of spaces  $J_n$ for $n=4,5$.
\label{fig: Jacobi spaces}}
\end{figure}

\subsection{Weighted count of rigid \PTCs}
\label{subsec: count of PTC}
In this section we study an enumerative problem of a weighted count of rational irreducible \PTCs\
with a fixed $\Delta$-set  $\Delta$ passing through a collection $\mathbf{p}\subset(\RR^2)^m$
of $m=n-1$ points in general position.
For this purpose we interpret a weighted count of curves as a generalized degree of the evaluation map.

Extend $\ev$ to a map
$$\widehat{\ev}:\widehat{\mathcal{M}}_{\Delta,n,n-1}\to(\RR^2)^{n-1}\cup\{\infty\}\cong\SS^{2n-2}$$
by $\widehat{\ev}(C_\infty)=\infty$.
Unfortunately, $\widehat{\ev}$ is not continuous:
\begin{ex}
\label{ex: discontinuous}
Let $C_t$, $t>0$ be a family of degenerate curves of Example \ref{ex: rigid} with
$\Delta=\{(-1,0),(-1,0),(2,0)\}$, $m=n-1=2$, both marked vertices $z_1$ and $v_*=z_2$
mapped to $h(z_1)=h(z_2)=(0,0)$ and the unmarked vertex $v$ mapped to $h(v)=(t,0)$.
We have $(l_1,l_2,h(v_*))=(t,t,(0,0))$ and $\ev(C_t)=((0,0),(0,0))\in (\RR^2)^2$, so
$C_t\to C_\infty$ as $t\to\infty$, but $\ev(C_t)\to ((0,0),(0,0))\ne\infty$.
\end{ex}

This, however, cannot happen if one considers only non-degenerate curves. Indeed,

\begin{lem}
\label{lem: ev_continuous}
Let $C_t$ be a family of curves in a top-dimensional face $\overline{\mathcal{M}}_{\Delta}(\mu)$
with $\mult(\mu, or_\mu)\ne 0$ such that $C_t\to C_\infty$ as $t\to\infty$. Then  $\ev(C_t)\to\infty$.
\end{lem}
\begin{proof}
By Proposition \ref{prop: Structure of evaluation map}, the restriction $\ev\restrict{\mu}$
of $\ev$ to $\overline{\mathcal{M}}_{\Delta}(\mu)$ is an invertible linear map, so
$||(\ev\restrict{\mu})^{-1}||\ne 0$.
A standard inequality $||Ax||\ge\frac{1}{||A^{-1}||}\cdot ||x||$ for an invertible linear operator
$A:\RR^d\to\RR^d$ shows that $Ax\to\infty$ as $x\to\infty$ and implies the lemma.
\end{proof}

Let $\widehat{\mathcal{M}}_{\Delta,n,n-1}^{reg}$ be the subset of
$\widehat{\mathcal{M}}_{\Delta,n,n-1}$ corresponding to all non-degenerate curves; its finite part
is a polyhedral subcomplex of $\mathcal{M}_{\Delta,n,n-1}$ generated by all top-dimensional faces
$\mathcal{M}_{\Delta}(\mu)$ with $\mult(\mu, or_\mu)\ne 0$.
On the space $\widehat{\mathcal{M}}_{\Delta,n,n-1}^{reg}$ there is a preferred orientation, namely,
the blackboard orientation on each of its  top-dimensional faces.
The Lemma above shows that the restriction of $\widehat{\ev}$ to
$\widehat{\ev}:\widehat{\mathcal{M}}_{\Delta,n,n-1}^{reg}\to\SS^{2n-2}$ is continuous.
Therefore, for any cycle $Z\in Z_{2n-2}(\widehat{\mathcal{M}}_{\Delta,n,n-1};\QQ)$ supported
on $\widehat{\mathcal{M}}_{\Delta,n,n-1}^{reg}$ its push-forward
$\widehat{\ev}_*(Z)\in Z_{2n-2}(\SS^{2n-2};\QQ)$
is well-defined and defines a pushforward class $\widehat{\ev}_*([Z])\in H_{2n-2}(\SS^{2n-2};\QQ)$.
Denote
$$\deg_Z(\ev):=I_{\ss^{2n-2}}(\widehat{\ev}_*([Z]),[pt])$$ the intersection number of
$\widehat{\ev}_*([Z])$ with the class $[pt]$ of a point in $\SS^{2n-2}$.

\begin{defn}
Let $\cD\subset (\RR^2)^{n-1}$ is the union $\cup_\mu\ev(\mathcal{M}_{\Delta}(\mu))$ over
all types $\mu$ such that either $\codim(\mu)>0$ or $\ev$ is not injective on
$\mathcal{M}_{\Delta}(\mu)$ (i.e., $\mult(\mu, or_\mu)=0$).
We say that an (ordered) collection $\mathbf{p}=(p_1,\dots,p_{n-1})\in (\RR^2)^{n-1}\subset\SS^{2n-2}$
of $n-1$ points in $\RR^2$ is in {\em general position}\footnote{
   Note that by Proposition \ref{prop: Structure of evaluation map} the set $\cD$ is of measure
   $0$ in $\SS^{2n-2}$; indeed, $\ev\restrict{\mu}$ is linear, so if it is not injective on some face,
   the dimension of its image drops at least by one.}
if $\mathbf{p}\notin\cD$.
\end{defn}

For such $\mathbf{p}$ denote by $N_{\Delta}(\mu,\mathbf{p})$ the (geometric) number
of rigid $(n-1)$-marked \PTCs\ of a combinatorial type $\mu$ with a $\Delta$-set  $\Delta$
through $\mathbf{p}$, i.e., with $h(z_i)=p_i$, $i=1,\dots,n-1$.
Note that $N_{\Delta}(\mu,\mathbf{p})$ equals ether 0 or 1 for any $\mu$; indeed, if
$\mult(\mu, or_\mu)=0$ then $N_{\Delta}(\mu,\mathbf{p})=0$  by definition of general
position and if $\mult(\mu, or_\mu)\ne 0$ then $\ev\restrict{\mu}$ is injective.

By the invariance of the homological intersection number and Corollary \ref{cor: orientation}
we finally obtain

\begin{thm}
\label{thm: count}
Let $\mathbf{p}\in(\RR^2)^{n-1}\minus\cD$ be a collection of $n-1$ points in $\RR^2$ in
general position and let $Z=\sum_{\mu}Z_\mu\cdot\widehat{\mathcal{M}}_{\Delta}(\mu)
\in Z_{2n-2}(\widehat{\mathcal{M}}_{\Delta,n,n-1};\QQ)$ be a cycle supported on
$\widehat{\mathcal{M}}_{\Delta,n,n-1}^{reg}$, equipped with the blackboard orientation.
Then the $Z$-weighted number
\begin{equation}
\label{eq: N_DeltaZ}
N_{\Delta,Z}(\mathbf{p}):=\sum_\mu Z_\mu\cdot N_{\Delta}(\mu,\mathbf{p})
\end{equation}
of curves through $\mathbf{p}$ does not depend on the collection $\mathbf{p}$ and
equals to $(-1)^n\deg_Z(\ev)$.
\end{thm}

\begin{rem}
\label{rem: automorphism count}
Note that this is the $Z$-weighted number of $(n-1)$-marked rational curves equipped with
an ordering of unmarked legs, which induces the ordering of the $\Delta$-set via the assignment
$\xi_i=\xi(e_i)$.
To get a more common count of curves with unordered legs (and unordered $\Delta$-set)
we should normalize $N_{\Delta,Z}$ by considering $\frac{1}{|Aut(\Delta)|}N_{\Delta,Z}$.
Here, $Aut(\Delta)$ is the automorphism set of $\Delta$, i.e., $\pi\in S_n$ such that
$\xi_{\pi(i)}=\xi_i$ for all $i=1,\dots,n$.
\end{rem}

In many cases $\widehat{\mathcal{M}}_{\Delta,n,n-1}^{reg}=\widehat{\mathcal{M}}_{\Delta,n,n-1}$,
so the requirement that the cycle $Z$ is supported on $\widehat{\mathcal{M}}_{\Delta,n,n-1}^{reg}$
is trivial. In particular, this happens if the $\Delta$-set is sufficiently generic. Namely, we shall call the
set $\Delta$ {\em 3-independent}, if for any partition of the index set $1,\dots,n$ into three
non-empty subsets $I_i$, $i=1,2,3$ vectors $\eta_i=\sum_{k\in I_i}\xi_k$, $i=1,2,3$ are not all parallel.
\begin{lem}
\label{lem: 3-independent}
Let $\Delta$ be 3-independent.
Then $\widehat{\mathcal{M}}_{\Delta,n,n-1}^{reg}=\widehat{\mathcal{M}}_{\Delta,n,n-1}$.
\end{lem}
\begin{proof}
Suppose that a curve $C$ in some top-dimensional cell $\overline{\mathcal{M}}_{\Delta,n,n-1}(\mu)$
is degenerate.
Then three slope vectors $\xi_v(e)$ at some trivalent unmarked vertex $v$ do not span $\RR^2$, i.e.
they are all parallel; it remains to notice that they can be written as  $\eta_i=\sum_{k\in I_i}\xi_k$,
$i=1,2,3$, where $I_i$ is the subset of legs in each of the connected component of $C\minus\{v\}$.
\end{proof}

\begin{ex}
\label{ex: rigid_count}
Consider curves of Example \ref{ex: rigid}.

For $m=n-1=1$ and a $\Delta$-set $\Delta=\{\xi,-\xi\}$ we have
$\widehat{\mathcal{M}}_{\Delta,2,1}^{reg}=\widehat{\mathcal{M}}_{\Delta,2,1}=\SS^2$,
$\widehat{\ev}:\SS^2\to\SS^2$ is the identity map and $\cD=\varnothing$.
For any $p\in\RR^2$, $N_{\Delta}(\mu,p)=1$ is the number of lines through $p$
in the direction $\xi$ and $N_{\Delta}(\mu,p)=\deg_Z(\ev)=1$.

For $m=n-1=2$ and a $\Delta$-set $\Delta=\{\xi_1,\xi_2,\xi_3\}$ with $\det(\xi_1,\xi_2)\ne 0$
we have $\widehat{\mathcal{M}}_{\Delta,3,2}^{reg}=\widehat{\mathcal{M}}_{\Delta,3,2}\cong
(\RR^2\times\RR^2)\cup\{\infty\}=\SS^4$, see
Example \ref{ex: tripod moduli}.
The set $\cD$ consists of $\mathbf{p}=(p_1,p_2)\in (\RR^2)^2$ such that $p_1-p_2$ is parallel to
$\xi_i$, $i=1,2,3$.
Pick the top-dimensional cycle $Z=\sum_{\mu_{ij}} \widehat{\mathcal{M}}_{\Delta}(\mu_{ij})$
(so $Z_\mu=1$ for all $\mu_{ij}$).  For any $\mathbf{p}\in(\RR^2)^2\minus\cD$,
$N_{\Delta,Z}(\mathbf{p})$ is then the (geometric) number of tripods with legs in directions
of $\Delta$ through $\mathbf{p}$.
The map $\widehat{\ev}:\widehat{\mathcal{M}}_{\Delta,3,2}\cong\SS^4\to\SS^4$ is a bijection
of degree $-1$, see Example \ref{ex: rigid}.
Thus $\deg_Z(\ev)=I_{\ss^{2n}}(\widehat{\ev}_*([Z]),[pt])=-1$ and for any
$\mathbf{p}\in(\RR^2)^2\minus\cD$ Theorem \ref{thm: count} implies that
$N_{\Delta,Z}(\mathbf{p})=(-1)^3\cdot(-1)=1$.
\end{ex}

\subsection{Counting caterpillar curves}
\label{subsec: caterpillar}
Let us illustrate the counting procedure based on Theorem \ref{thm: count} on an example
of caterpillar curves.

Consider a top-dimensional face $\overline{\mathcal{M}}_{\Delta}(\mu)$ of
$\widehat{\mathcal{M}}_{\Delta,n,n-1}$.
A trivalent tree $\G$ of a combinatorial type $\mu$ is an {\em $(s,t)$-caterpillar}, $1\le s<t\le n$,
if all unmarked vertices can be connected by a non-self-intersecting path of edges $L_{st}(\G)$
(the ``body'' of the caterpillar), starting at the leg  $e_s$ and ending at the leg $e_t$
(the ``head'' and the ``tail'' of the caterpillar).
To assure that any such face belongs to $\widehat{\mathcal{M}}_{\Delta,n,n-1}^{reg}$,
a weaker version of 3-independence will suffice:
we say that the $\Delta$-set $\Delta$ is $(s,t)$-independent, if for any partition of the index
set $1,\dots,n$ into three non-empty subsets $I_i$, $i=1,2,3$ with $s\in I_1$,  $t\in I_2$
and $|I_3|=1$ vectors $\eta_i=\sum_{k\in I_i}\xi_k$, $i=1,2,3$ are not all parallel.
Let us also define the sign of an $(s,t)$-caterpillar curve $C$ of type $\mu$ as follows.
We follow the polygonal path $h(L_{st}(\G))$ of edges in $h(\G)$ starting at $h(e_s)$
and ending at $h(e_t)$ and count unmarked edges approaching it from the right.
We set $\eps_{st}(C)=+1$ if this number is even, and $\eps_{st}(C)=-1$ if it is odd.
An example of signs $\eps_{st}$ for $H$, $X$ and $I$ plane curves is shown in Figure
\ref{fig: caterpillar}.
%We have $\eps_{13}(H)=\eps_{24}(H)=-1$, $\eps_{14}(H)=\eps_{23}(H)=+1$,
%$\eps_{12}(X)=\eps_{34}(X)=-1$, $\eps_{14}(X)=\eps_{23}(X)=+1$,
%and $\eps_{13}(I)=\eps_{24}(I)=-1$, $\eps_{12}(I)=\eps_{34}(I)=+1$.

\begin{figure}[htb]
\includegraphics[width=5.0in]{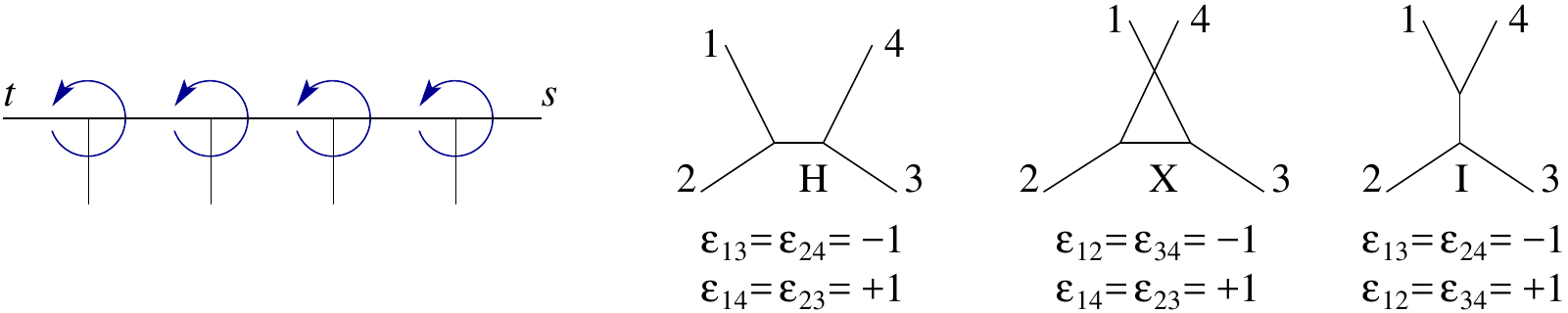}
\caption{Abstract and plane caterpillar curves.
\label{fig: caterpillar}}
\end{figure}

\begin{prop}

Let $\Delta$ be an $(s,t)$-independent $\Delta$-set for some $1\le s<t\le n$ and let
$\mathbf{p}\in(\RR^2)^{n-1}\minus\cD$ be a collection of $n-1$ points in $\RR^2$ in general position.
Then the algebraic number of rigid $(s,t)$-caterpillar curves in $\widehat{\mathcal{M}}_{\Delta,n,n-1}$
through $\mathbf{p}$ (i.e., with marked points at $\mathbf{p}$), counted with signs $\eps_{st}$, does
not depend on $\mathbf{p}$.
\end{prop}
\begin{proof}
Define a $(2n-2)$-chain $Z_{st}=\sum_\mu Z_{st}(\mu)\cdot\widehat{\mathcal{M}}_{\Delta}(\mu)$
with $Z_{st}(\mu)=1$ if $\mu$ is a combinatorial class represented by an $(s,t)$-caterpillar
(with the orientation shown in Figure \ref{fig: caterpillar}) and $Z_{st}(\mu)=0$ otherwise.
Note that if $\Delta$ is $(s,t)$-independent, then $Z_{st}$ is supported on $\mathcal{M}_{\Delta,n,n-1}^{reg}$,
similarly to the proof of Lemma \ref{lem: 3-independent}.
The sign $\eps_{st}$ of a caterpillar curve simply indicates whether the orientation on Figure
\ref{fig: caterpillar} is the same as (or opposite to) the blackboard orientation.

The chain $Z_{st}$ is a cycle. Indeed, we should check that coefficients $Z_{st}(\mu)$ satisfy
(IHX) and (MP) equations. But $Z_{st}(\mu)$ do not depend on the position of marked points,
so satisfy the (MP) equation. Moreover, it is easy to see that either none or exactly two trees
appearing in (IHX) are $(s,t)$-caterpillars; a comparison of their signs $\eps_{st}$ in Figure
\ref{fig: caterpillar} shows that the (IHX) equation is also satisfied.
\end{proof}

\begin{ex}
\label{ex: caterpillar cycles}
Consider $n=5$.
Recall that $H_{8}(\widehat{\mathcal{M}}_{\Delta,5,4};\QQ)\cong J_5$ and the combinatorics
of $J_5$ is described by the Petersen graph, see Example \ref{ex: Jacobi spaces}.
Cycles $Z_{12}$,  $Z_{13}$, $Z_{14}$ and $Z_{15}$ are shown in Figure \ref{fig: caterpillar cycles}.
It is easy to check that cycles $Z_{st}$ generate $H_{8}(\widehat{\mathcal{M}}_{\Delta,5,4};\QQ)\cong J_5$.
\begin{figure}[htb]
\includegraphics[width=5.0in]{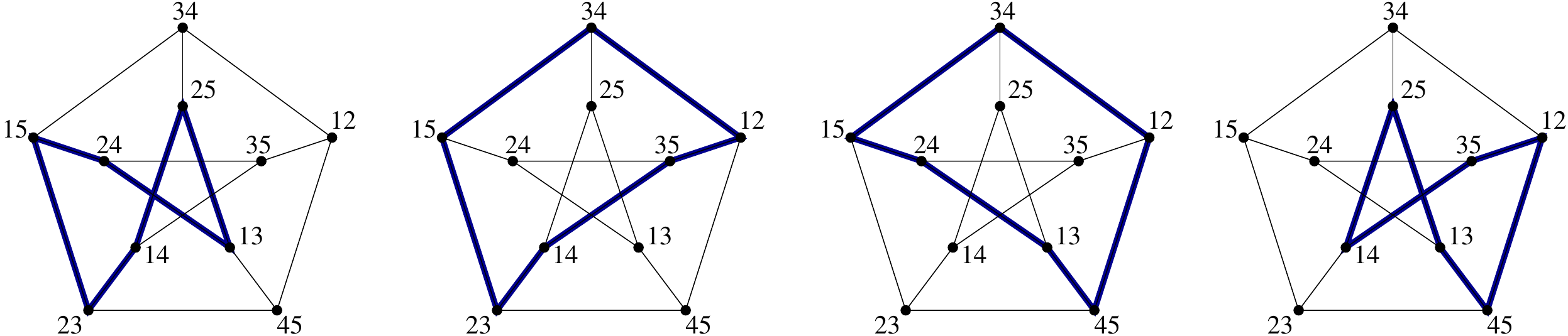}
\caption{Caterpillar cycles $Z_{12}$,  $Z_{13}$, $Z_{14}$ and $Z_{15}$ in  $J_5$.
\label{fig: caterpillar cycles}}
\end{figure}
 \end{ex}

\section{Lie algebras and enumeration of rational curves}
\label{sec: Lie and enumeration}

\subsection{Top-dimensional cycles and Lie algebras}
\label{subsec: cycles}
The  space $\widehat{\mathcal{M}}_{\Delta,n,n-1}$ contains many interesting cycles in
$H_{2n-2}(\widehat{\mathcal{M}}_{\Delta,n,n-1};\QQ)\cong J^{rig}_{n,n-1}\cong J_{n-1}$
(and in its complexification).
The space $J_{n-1}$ can be viewed as a diagrammatic version of a Lie algebra, so cycles are
closely related to the so-called weight systems arising from Lie algebras.
We shall skip this construction in the general case, specializing instead to a family $\FG_\hbar$
of Lie algebras suitable for the enumeration of rational \PTCs\ -- known as the quantum sine,
trigonometric, or quantum tori Lie algebras.
Initiated by Arnold \cite{A} in the classical case $\hbar=0$, they were extensively studied over
the years both by mathematicians and physicists, see e.g. \cite{FFZ}.
Lately they reappeared e.g. in \cite{KS}.

To define $\FG_\hbar$ we first define a quantum number $\displaystyle [x]_\hbar$ for
$\hbar \in\CC\minus 2\cdot\ZZ$ and $x\in\RR$ by
\begin{equation}
\label{eq: quantum numbers}
\displaystyle [x]_\hbar=\frac{e^{\pi i\hbar x/2}-e^{-\pi i\hbar x/2}}{e^{\pi i\hbar/2}-e^{-\pi i\hbar/2}}=
\frac{\sin(\pi\hbar x/2)}{\sin(\pi\hbar/2)}\,.
\end{equation}
For $\hbar=0$ we have $\lim\limits_{\hbar\to 0} [x]_\hbar=x$ for any $x\in\RR$, so we define\footnote{
For $\hbar=2k$, $k\ne 0$ the situation is different. We still have
$\lim\limits_{\hbar\to 2k} [x]_\hbar=x$ for any $x\in\ZZ$.
However, for non-integer $x$'s the function $\hbar\mapsto [x]_\hbar$ has poles at $\hbar=2k$, $k\ne 0$. }
$[x]_0=x$.

Now, let $A$ be the Abelian group generated by $\xi_i\in\Delta$ with the antisymmetric bilinear form
$A\times A\to\RR$ given by $\alpha\times \beta:=\det(\alpha,\beta)$.
Let $\FG_\hbar$ be a Lie algebra  generated by $\{L_\alpha\}_{\alpha\in A}$ and the Lie bracket
$\displaystyle [L_\alpha,L_\beta]=\sum_\gamma c_{\alpha \beta}^\gamma L_\gamma$ with
$c_{\alpha \beta}^\gamma=[\alpha\times \beta]_\hbar\cdot \delta_{\alpha+\beta}^{\,\gamma}$, i.e.,
\begin{equation}
\label{eqn: Lie}
[L_\alpha,L_\beta]=[\alpha\times\beta]_\hbar\ L_{\alpha+\beta}\,.
\end{equation}
Here, the identity $[-x]_\hbar=-[x]_\hbar$ implies antisymmetry
and the trigonometric formula $2\sin(a)\sin(b)=\cos(a-b)-\cos(a+b)$
implies the Jacobi identity for the bracket:
\begin{equation}
\label{eqn: AS_Lie}
c_{\beta\alpha}^\gamma=-c_{\alpha \beta}^\gamma\,,
\end{equation}
\begin{equation}
\label{eqn: IHX_Lie}
\sum_\nu\left(
c_{\alpha \beta}^\nu \cdot c_{\nu\gamma}^\tau-
c_{\alpha\gamma}^\nu \cdot c_{\nu\beta}^\tau-
c_{\beta\gamma}^\nu \cdot c_{\alpha\nu}^\tau\right)=0.
\end{equation}
It is also convenient to define a symmetric tensor $t^{\alpha\beta}=\delta^{\alpha+\beta,0}$
in order to raise indices of structural constants:
$c_\alpha^{\beta \gamma}:=\sum_\nu c_{\alpha\nu}^\beta t^{\nu\gamma}$.
Using the definition of $c_{\alpha \beta}^\gamma$ it is easy to check that
$c_\alpha^{\beta \gamma}$ are antisymmetric w.r.t upper indices:
\begin{equation}
\label{eqn: MP_Lie}
c_\alpha^{\gamma \beta}=-c_\alpha^{\beta \gamma}.
\end{equation}

An $A$-coloring (or ``$A$-state'') $s$ of an $(n,n-1)$-tree $\G$ of type $\mu$
(equipped with a Lie-orientation and an outer flow directions of edges) is an assignment
$s:E\to A$; an $A$-coloring is compatible with $\Delta$, if each unmarked leg $e_i$ of
$\G$ is colored by $\xi_i$ for $i=1,\dots,n$ and each marked leg is colored by $0$.
Fix an ordering $e_1(v)$, $e_2(v)$ of the incoming half-edges at each unmarked vertex
$v$, consistent with the Lie-orientation of $\G$.
The local weight $w(v;s)$ of an unmarked trivalent vertex $v$ with the incoming half-edges
$e_1(v)$, $e_2(v)$ colored by $\alpha$ and $\beta$ and the outgoing half-edge colored by
$\gamma$ is defined by $w(v;s)=c_{\alpha \beta}^\gamma$.
The local weight $w(v;s)$ of a marked trivalent vertex $v$ with outgoing unmarked edges colored
by $\alpha$ and $\beta$ is defined by $w(v;s)=t^{\alpha\beta}$.
The global weight $W(\G;s)$ of an $A$-coloring $s$ of $\G$ is defined by taking the product of local
weights over all vertices of $\G$:
$$ W(\G;s)=\prod_v w(v;s)\,.$$
In particular, for $m=n-1=1$ a 1-marked line, i.e., a rational 1-marked \PTC\ with the $\Delta$-set
$\Delta = \{\xi,-\xi\}$ has only one marked vertex of weight $t^{\xi(-\xi)}=1$, so its weight is $1$
independently on the vector $\xi$.
Note that due to our choice of the structural constants $c_{\alpha \beta}^\gamma$ and
$t^{\alpha\beta}$ there is only one $A$-coloring $s_\Delta$ of $\G$ compatible with $\Delta$
which may give a non-zero weight $W(\G;s)$, namely the one which defines a balanced global
current $\widetilde{\xi}$ on $\G$, see Corollary \ref{cor: Uniqueness of extension} and
Figure \ref{fig: colorings}.
\begin{figure}[htb]
\includegraphics[width=2.0in]{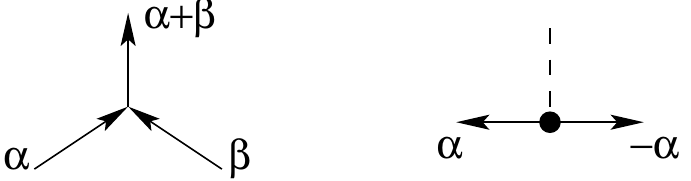}
\caption{Colorings around unmarked and marked vertices.
\label{fig: colorings}}
\end{figure}

Finally, the weight of $\G$ is defined as the formal sum $ Z_{\Delta,\hbar}(\G)= \sum\limits_s W(\G;s)$
over all $A$-colorings compatible with $\Delta$; thus $Z_{\Delta,\hbar}(\G)=\displaystyle W(\G;s_\Delta)$.
Since $Z_{\Delta,\hbar}(\G)$ depends only on the combinatorial type $\mu$ and the Lie-orientation
$or_\mu$ of $\G$, we shall also denote it by $Z_{\Delta,\hbar}(\mu, or_\mu)$.
Note that $Z_{\Delta,\hbar}(\mu, -or_\mu)=-Z_{\Delta,\hbar}(\mu, or_\mu)$.
Also, since the structural constants $c_{\alpha \beta}^\gamma$ are defined via the cross product of vectors,
the weight $Z_{\Delta,\hbar}(\mu, or_\mu)=0$ for any degenerate $\mu$, i.e., for $\mult(\mu, or_\mu)=0$.

Extending $Z_{\Delta,\hbar}$ by linearity to the $\CC$-linear function on the $\QQ$-vector space
spanned by $(n,n-1)$-trees we get

\begin{prop}
\label{prop: Z_hbar}
The function $\displaystyle Z_{\Delta,\hbar}$ factors through the {\rm (AS)}, {\rm (MP)} and
{\rm (IHX)} relations, thus descends to a linear function
$\displaystyle Z_{\Delta,\hbar}:J^{rig}_{n,n-1}\to\CC$.
\end{prop}
\begin{proof}
The weights $Z_{\Delta,\hbar}$ of graphs satisfy the (AS) relation due to \eqref{eqn: AS_Lie}.

The weights of shown fragments of the colored H, X, I graphs of Figure \ref{fig: IHX_MP}a
below are
$\sum_\nu c_{\alpha \beta}^\nu \cdot c_{\nu\gamma}^\tau$,
$\sum_\nu c_{\gamma\alpha}^\nu \cdot c_{\beta\nu}^\tau$ and
$\sum_\nu c_{\beta\gamma}^\nu \cdot c_{\alpha\nu}^\tau$
respectively, thus  satisfy the (IHX) relation due to \eqref{eqn: IHX_Lie}.

The weights of shown fragments of the colored marked H, X graphs of Figure \ref{fig: IHX_MP}b
below are
$\sum_\nu c_{\alpha\nu}^\gamma t^{\nu\beta}=c_\alpha^{\gamma\beta}$ and
$\sum_\nu c_{\nu\alpha}^\beta t^{\gamma\nu}=-c_\alpha^{\beta\gamma}$
respectively, thus satisfy the (MP) relation due to \eqref{eqn: MP_Lie}.
\end{proof}

\begin{figure}[htb]
\includegraphics[width=5.0in]{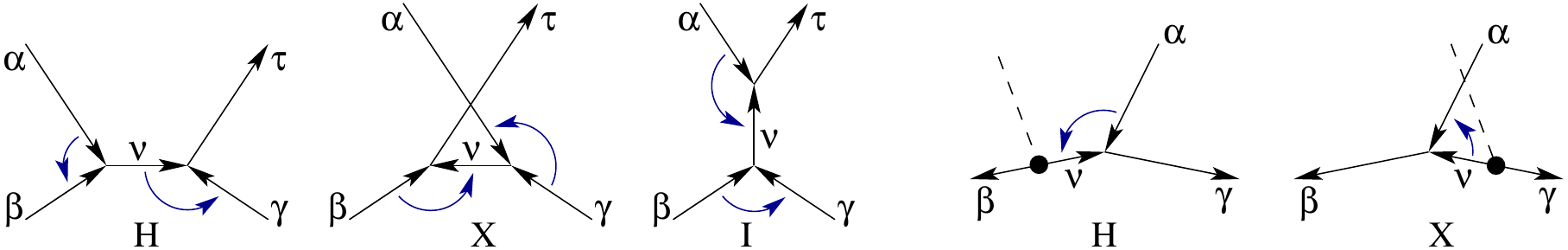}
\\ \hspace{0.55in} a\hspace{2.45in}b
\caption{Colored I,H,X and marked H,X graphs.
\label{fig: IHX_MP}}
\end{figure}

\begin{rem}
In view of the forgetful isomorphism $f:J^{rig}_{n,n-1}\to J_{n-1}$ the value
$Z_{\Delta,\hbar}(\mu, or_\mu)$ may be calculated as follows: drop all markings of an
$(n,n-1)$-tree $\G$ of type $\mu$ and direct all edges of the resulting unmarked Lie-oriented
tree $f(\G)$ towards the last $n$-th leg (the ``root''), see Figure \ref{fig: forgetful map}a.
Color incoming legs $e_i$, $i=1,\dots,n-1$ with vectors $-\xi_i$ (and the outgoing leg $e_n$
with $\xi_n=-\sum_{i=1}^{n-1} \xi_i$), extending it uniquely to other edges using the balancing
condition, see Figure \ref{fig: colorings}a.
The global weight of $f(\G)$ is again defined as the product of local weights
$c_{\alpha\beta}^\gamma$ over all trivalent vertices of $f(\G)$; note that since all vertices of
$f(\G)$ are unmarked, we don't use the tensor $t^{\alpha\beta}$.
By construction this weight equals to the weight $\displaystyle W(\G_*;s_\Delta)$
of the canonical representative $\G_*$ of the equivalence class of $\G$ in $J^{rig}_{n,n-1}$,
see Figure \ref{fig: forgetful map}b and Lemma \ref{lem: forgetful map} (and thus equals to
the weight $Z_{\Delta,\hbar}(\G)=\displaystyle W(\G;s_\Delta)$ of $\G$).
\end{rem}

Let $\cZ_\hbar$ be a $\CC$-valued $(2n-2)$-chain
$\cZ_\hbar=\sum_\mu Z_{\Delta,\hbar}(\mu, or_\mu)\cdot\widehat{\mathcal{M}}_{\Delta}(\mu)$
in $\widehat{\mathcal{M}}_{\Delta,n,n-1}$.
Using the identification of $J^{rig}_{n,n-1}$ with $H_{2n-2}(\widehat{\mathcal{M}}_{\Delta,n,n-1};\QQ)$,
see Theorem
\ref{thm: Jacobi homology}, we obtain
\begin{cor}
The chain $\cZ_\hbar$ is a cycle with
% its class
$[\cZ_\hbar]\in H_{2n-2}(\widehat{\mathcal{M}}_{\Delta,n,n-1};\QQ)\otimes\CC$.
The cycle $\cZ_\hbar$ is supported on $\widehat{\mathcal{M}}_{\Delta,n,n-1}^{reg}$
(i.e., on faces, corresponding to non-degenerate types $\mu$).
If $\widehat{\mathcal{M}}_{\Delta,n,n-1}^{reg}$ is equipped with the blackboard orientation,
the cycle  $\cZ_\hbar$ is positive, i.e. $Z_{\Delta,\hbar}(\mu, or_\mu)>0$ for any $\mu$ in
$\cZ_\hbar$.
\end{cor}

In view of Theorem \ref{thm: count} we therefore conclude
\begin{cor}
Let $\mathbf{p}\in(\RR^2)^{n}\minus\cD$ be a collection of $n-1$ points in $\RR^2$ in
general position.
The $\cZ_\hbar$-weighted number $N_{\Delta,\cZ_\hbar}=N_{\Delta,\cZ_\hbar}(\mathbf{p})$
of curves through $\mathbf{p}$ does not depend on the collection $\mathbf{p}$ and equals to
$(-1)^n\deg_{\cZ_\hbar}(\ev)$.
\end{cor}

\begin{rem}
Some authors (see \cite{FFZ}) consider the Lie algebra structure on $\FG_\hbar$  given by
$ [L_\alpha,L_\beta]=r\sin\left(\pi\hbar\frac{\alpha\times\beta}2\right)\cdot L_{\alpha+\beta}$
for arbitrary complex constant $r$. Our choice $r=\sin(\pi\hbar/2)^{-1}$ gives the usual $q$-numbers
so $[L_\alpha,L_\beta]=(\alpha\times\beta)\cdot L_{\alpha+\beta}$ in the limit $\hbar\to 0$.
Other values of $r$ can be obtained by a suitable rescaling $L_\alpha\to r L_\alpha$ of the basis
and result in the same cycle $\cZ_\hbar$ up to multiplication by a constant.
\end{rem}

In the following sections we shall study $N_{\Delta,\cZ_\hbar}$ for $\hbar=0$.

\subsection{Quasi-classical limit of weights}
\label{subsec: quasi-classics}
In the quasi-classical limit $\hbar\to 0$ the only non-trivial structural constants of the Lie algebra
$\FG_0$ are $c_{\alpha\beta}^{\alpha+\beta}=[\alpha\times\beta]_0=\alpha\times\beta$ and
$c_{\alpha+\beta}^{\alpha\beta}=[\beta\times\alpha]_0=\beta\times\alpha$.

The Jacobi identity \eqref{eqn: IHX_Lie} becomes in this case the quadratic Pl\"ucker relation in
the real Grassmanian $Gr(2,4)$.
Indeed, note that for $\tau=\alpha+\beta+\gamma$ vectors $\alpha$, $\beta$, $\gamma$ and
$-\tau$ can be arranged into a $2\times 4$ real matrix such that the sum of entries in each
row is zero. Denote by $p_{ij}$ its $(i,j)$-th minor; since the sum of entries in each row is zero,
we have
$p_{13}+p_{23}-p_{34}=p_{12}+p_{32}-p_{24}=p_{12}+p_{13}+p_{14}=0$.
Then \eqref{eqn: IHX_Lie} becomes
\begin{equation}
\label{eqn: Plucker}
p_{12}(p_{13}+p_{23})- p_{13}(p_{12}+p_{32})- p_{23}(p_{12}+p_{13})=
p_{12}p_{34}- p_{13}p_{24}+ p_{23}p_{14}=0\,.
\end{equation}

It remains to notice that the matrix above determines a 2-plane in $\RR^4$ lying in the subspace
orthogonal to the vector $(1,1,1,1)$. Denote by $Gr^0(2,4)$ the subvariety of $Gr(2,4)$, consisting
of such 2-planes. The equation \eqref{eqn: Plucker} is the standard Pl\"ucker relation on $Gr^0(2,4)$.

Structural constants of the Lie algebra $\FG_\hbar$ provide an $\hbar$-deformation
$p_{ij}\mapsto[p_{ij}]_\hbar$ of the Pl\"ucker coordinates $p_{ij}$ on $Gr^0(2,4)$.
A natural question is whether there exist other continuous deformations of the Pl\"ucker
coordinates on $Gr^0(2,4)$ of the form $p_{ij}\mapsto \varphi(p_{ij})$.
The next proposition shows that up to linear reparametrizations any such deformation
coinsides with $p_{ij}\mapsto[p_{ij}]_\hbar$:

\begin{prop}
Let $\varphi:\RR\to\CC$ be a non-linear continuous function.
Then it preserves the Pl\"ucker relation on the subvariety $Gr^0(2,4)$ if and only if it has the form
$\varphi(x)=r\sin(kx)$ with $r,k\in\CC$.
\end{prop}

\begin{proof}
In view of \eqref{eqn: Plucker}, we are looking for the continuous solutions to the functional equation
$\varphi(x)\varphi(y+z)-\varphi(y)\varphi(x-z)-\varphi(z)\varphi(x+y)=0$ for all $x,y,z\in\RR$.
Substituting $y=z$, we get the following functional equation:
$\varphi(x)\varphi(2y)=\varphi(y)(\varphi(x+y)+\varphi(x-y))$ for all $x,y\in\RR$.
By \cite[Theorem 2.7]{ParVas} the only continuous solutions of this equation are $\varphi(x)= kx$ or
$\varphi(x)= r\sin(kx)$ for some $r,k\in\CC$.
\end{proof}

\subsection{Dependence on the $\Delta$-set}
\label{subsec: eps-dependence}
 Note that coefficients $Z_{\Delta,\hbar}(\mu, or_\mu)$
of $\cZ_\hbar$ smoothly depend on vectors $\xi_i$, $i=1,\dots,n$.
This leads to various interesting consequences.

Let us illustrate this on an example of $(s,t)$-caterpillar curves which we considered in Section \ref{subsec: caterpillar}.
Take $\hbar=0$ and fix a pair $(s,t)$ of indices $1\le s< t\le n$.
Consider a deformed $\Delta$-set
$$\Delta_\eps=\{\xi_1,\dots,(1+\eps)\xi_s,\dots,\xi_t-\eps\xi_s,\dots,\xi_n\}$$
depending on a parameter $\eps\in\RR$.
Then each $Z_{\Delta_\eps,0}(\mu, or_\mu)$ is polynomial in $\eps$ of a degree at most $n-2$ and
we may consider terms of each degree separately. In particular, from the construction of $\cZ_0$ it is
easy to see that the leading coefficient of $\eps^{n-2}$ is non-zero only if a tree $\G$ representing
$\mu$ is an $(s,t)$-caterpillar; for any such $\G$ (with the blackboard orientation) the coefficient is $a_{2n-2}=\prod_i\det(\xi_i,\xi_s)$ where the product is over all $i\ne s,t$.
Thus, up to a multiplication by $a_{2n-2}$, top degree terms of $\cZ_0$ coincide with the
$(s,t)$-caterpillar cycle $Z_{st}$ of Section \ref{subsec: caterpillar}.

\subsection{Relation to classical enumerative problems}
\label{subsec: classical enumeration}
Let us consider the case of genuine rational tropical curves on a tropical toric surface.
This corresponds to a $\Delta$-set $\Delta=\{\xi_i\}$,  $i=1,2,\dots,n$, such
that $\xi_i\in\ZZ^2$.

For such a $\Delta$-set, \PTCs\ become genuine tropical curves after two minor modifications.
Firstly, we should drop the ordering of legs.
Secondly, we should add a notion of primitive integer vectors $\xi_{pr}(e)$ along edges
and corresponding multiplicities $m(e)\in\ZZ$ of edges: vectors $\xi(e)$ should be written
as a product $m(e)\xi_{pr}(e)$ with $m(e)$ being the GCD of the coordinates of $\xi(e)$.

Now we can compare our weighted number $N_{\Delta,\hbar}$ with the refined count of
rational tropical curves introduced by Block and G\"ottsche \cite{BG} (for the invariance of 
this count and its comparison to the usual count of complex and real tropical curves see
\cite{IM}). There are slight differences in notations and normalizations. 
Namely, we use a slightly different notation for quantum numbers: 
$[x]_y$ in notations of \cite{BG} equals our $[x]_\hbar$ for $y=e^{\pi i\hbar}$.
Also, in order to count curves with unordered legs as in \cite{BG} we should consider 
the normalized number $\frac1{| Aut(\Delta)|}N_{\Delta,\cZ_\hbar}$, see Remark 
\ref{rem: automorphism count}.
Apart from that, the only non-trivial difference in the setup is the presence of an 
orientation of the moduli space $\widehat{\mathcal{M}}_{\Delta,n,n-1}$ and an 
ordering of incoming half-edges in unmarked vertices (consistent with the orientation 
of $\widehat{\mathcal{M}}_{\Delta,n,n-1}$) required in the definition of vertices' weights 
$[\alpha\times\beta]_\hbar$, while  Block and G\"ottsche \cite{BG} consider the weight 
$[|\alpha\times\beta|]_\hbar$ corresponding to an unordered pair of edges.

To compare signs involved in these formulas for a graph $\G$ of a combinatorial type
$\mu$, fix an ordering of the incoming half-edges $e_1(v)$, $e_2(v)$ incident to $v$ at
each unmarked vertex $v$ of $\G$, consistent with the orientation $or_\mu$.
Note that by the definition \eqref{eq: quantum numbers} of quantum numbers we have
$[\alpha\times\beta]_\hbar=\sign(\det(\alpha,\beta))\cdot[|\alpha\times\beta|]_\hbar$.
Thus, using the definition \eqref{eq: mult} of $\mult(\mu, or_\mu)$, we may rewrite the  global
weight $Z_{\Delta,\hbar}(\mu)=W(\G;s_\Delta)=\prod_v [\xi_v(e_1(v))\times\xi_v(e_2(v)]_\hbar$
as  $$Z_{\Delta,\hbar}(\mu)=
\sign(\mult(\mu, or_\mu))\cdot\prod_v[|\xi_v(e_1(v))\times\xi_v(e_2(v)|]_\hbar\,,$$
where the product is over all unmarked vertices of $\G$, see Equation \eqref{eq: mult}.
In particular, for the blackboard orientation % of $\mathcal{M}_{\Delta}(\mu)$
we have $\sign(\mult(\mu, or_\mu))=+1$, thus
$Z_{\Delta,\hbar}(\mu)=\prod_v[|\xi_v(e_1(v))\times\xi_v(e_2(v)|]_\hbar$.
Therefore we can conclude that the normalized $\cZ_\hbar$-weighted number
$\frac1{| Aut(\Delta)|}N_{\Delta,\cZ_\hbar}(\mathbf{p})$ of curves through $\mathbf{p}$
coincides with the corresponding Block-G\"ottsche's refined count of rational tropical curves.
\begin{rem}
Recall that the refined count of \cite{BG} interpolates between the tropical Gromov-Witten 
invariant (counting complex curves) for $y=e^{\pi i\hbar}=1$ and the tropical Welschinger invariant 
(counting real curves) for $y=e^{\pi i\hbar}=-1$, see \cite[Theorem 1]{IM}. 
Thus in the tropical case $\frac1{| Aut(\Delta)|}N_{\Delta,\cZ_\hbar}$ 
equals to the tropical Gromov-Witten invariant for $\hbar = 0$ and to the tropical Welschinger 
invariant for $\hbar = 1$.  
If we denote the tropical complex multiplicity $|\xi_v(e_1(v))\times\xi_v(e_2(v))|$ of $v$ by 
$w_{\cc}(v)$, then the real multiplicity $w_{\rr}(v)$ of $v$ is usually defined to be 0 
if $w_{\cc}(v)$ is even and  $(-1)^{\frac{w_{\cc}(v)-1}{2}}$ if $w_{\cc}(v)$ is odd. 
In the general pseudotropical case, $w_{\cc}(v)$ is not necessarily an integer; thus 
$w_{\rr}(v)=[w_{\cc}(v)]_1=\sin(\pi\cdot w_{\cc}(v)/2)$ may take non-integer values 
between $-1$ and $1$. 
This sheds a new light on the real multiplicity $w_{\rr}(v)$.
\end{rem}

\section{Recursive formula}
\label{sec: recursion}
In this section we shall derive a recursive formula for the numbers $N_{\Delta,\hbar}$.
Since we shall be now interested only in its dependence on the vectors in $\Delta$, let us
fix $\hbar$ and use a shorthand notation $N(\Delta):=N_{\Delta,\hbar}$.

\subsection{Notations and definitions}
For $n>3$ fix a $\Delta$-set $\Delta = \{\xi_1,\ldots, \xi_n\}$ and a configuration
$\mathbf{p} = (p_1,\dots, p_{n-1})\in(\RR^2)^{n-1}\minus\cD$ of $n-1$ points in
$\RR^2$ in general position.
We shall take one of the points -- say, $p_{n-1}$ -- close to infinity in certain direction, while
keeping all other marked points in a small disk $D$ around the origin.
To describe the structure of curves through $\mathbf{p}$ we shall need some notations.

Let $C =[(\G, \ell, h, n-1)]$ be a curve in a top-dimensional face $\mathcal{M}_{\Delta}(\mu)$.
Two connected components of $\G\minus E^m$ incident to $z_{n-1}$ contain two legs of $\G$;
denote them by $s$ and $t$, so that $\det(\xi_s,\xi_t)\ge0$.
Denote also by $L$ a path of edges of $\G$ starting at $s$ and ending on $t$, numbering
unmarked vertices $v_1,v_2,\dots,v_k$ on $L$ in the order of their passage along $L$.
The complement $\G\minus L$ of $L$ consists of $k$ connected components; denote
by $\G_j$, $j=1,\dots,k$ a component meeting $L$ at $v_j$. See Figure \ref{fig: L_cut}.
Let $I_i$ be the set of indices of all legs of $\G$ in $\G_j$; denote $n_j:=|I_j|$.
Note that $\sum n_j=n-2$ and together $s$, $t$ and subsets $(I_1,\dots,I_k)$ define a partition
$S$ of $\{1,\dots,n\}$.
Equip $\G_j$ with Lie-orientations at vertices, slope vectors, lengths of edges and ordering of legs
inherited from $\G$. Consider the edge of $\G_j$ meeting $L$ at $v_j$ as the last, $(n_j+1)$-st,
leg of infinite length with the slope vector $\eta_j := - \sum_{i\in I_j} \xi_i$.
This data makes $\G_j$ into an $n_j$-marked curve $C_j =[(\G_j, \ell\,\restrict{\G_j}, h\,\restrict{\G_j}, n_j)]$.
The outer flow directions of edges inherited from $\G$ implies that all curves $C_j$ are rigid.
See Figure \ref{fig: L_cut}.

\begin{figure}[htb]
\includegraphics[width=5.0in]{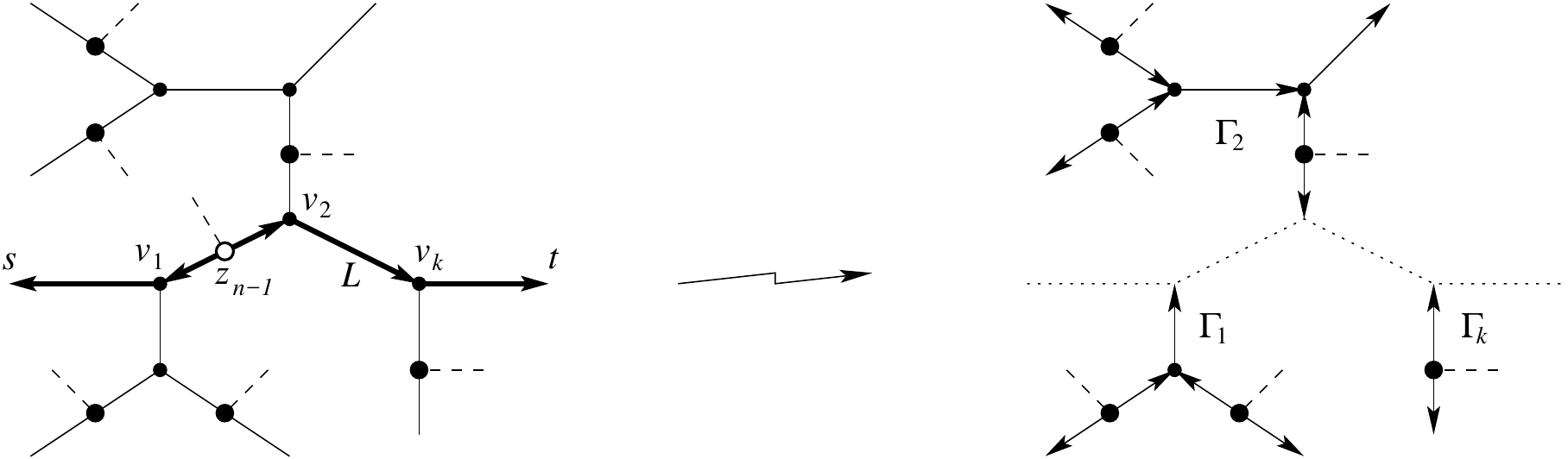}
%\vspace{1in}
\caption{Cutting $\G$ along the path $L$.
\label{fig: L_cut}}
\end{figure}

\begin{lem} Suppose that all marked points $p_i$, $i=1,\dots, n-2$ (except for $p_{n-1}$)
lie in a small disk of radius $\eps$ around the origin.
Let $C\in\ev^{-1}(\mathbf{p})$ be a curve of a combinatorial type $\mu$ with
$\mult(\mu, or_\mu)\ne 0$.
Then there is a constant $c_\mu\in\RR$ (depending only on $\mu$) such that all vertices
of $h(\G\minus L)$ lie in a disk of radius $c\cdot\eps$ around the origin.
\end{lem}

\begin{proof}
We repeat the argument of Lemma \ref{lem: ev_continuous}.
Each vertex of $\G\minus L$ belongs to $\G_j$ for some $j=1,\dots,k$.
Denote by $\mu_j$ the combinatorial type of $\G_j$.
The rigid marked curve $\G_j$ is non-degenerate, since $\mult(\mu, or_\mu)\ne 0$ so at any
vertex $v$ of $\G$ (and thus, in particular, of $\G_i$) slope vectors $\xi_v(e)$ span $\RR^2$.
Thus by Proposition \ref{prop: Structure of evaluation map} the restriction $\ev\restrict{\mu_j}$
of $\ev$ to $\overline{\mathcal{M}}_{\Delta}(\mu_j)$ is an invertible linear map, so
$\left|\left|(\ev\restrict{\mu_j})^{-1}\right|\right|\ne 0$.
Finally, the inequality $||x||\le||A^{-1}||\cdot ||Ax||$ for an invertible linear operator
$A:\RR^d\to\RR^d$ implies the lemma for
$c_\mu=\max\limits_j \left|\left|(\ev\restrict{\mu_j})^{-1}\right|\right|$.
\end{proof}

\subsection{Taking a marked point to infinity}
\label{subsec: p_to_infinity}
Now, take $p_{n-1}$ close to infinity in some sufficiently generic (different from directions of
vectors $\sum_{i\in I}\xi_i$ for all subsets $I$ of $1,2,\dots,n$) direction $\xi_0$, while keeping
all other marked points in a small disk around the origin.
Let $C\in\ev^{-1}(\mathbf{p})$ be a non-degenerate curve through $\mathbf{p}$.
By the above Lemma, enlarging the initial disk we can assure that all (both marked and unmarked)
vertices of $h(\G\minus L)$ lie in a small disk $D$. Due to our generic choice of direction $\xi_0$
and the fact that $p_{n-1}$ is close to infinity we may also assume that the pair of edges in $L$
incident to $z_{n-1}$ map outside $D$.
Then the image $h(\G)$ of $C$ looks as shown in Figure \ref{fig: pt2infty}:
$p_{n-1}$ lies on $h(L)$, with all vertices of $h(\G\minus L)$ being inside $D$ and $k$ edges in
$h(\G\minus L)$ starting in $D$ and ending in vertices $h(v_j)$ on $h(L)$.
The balancing condition implies that the path $h(L)$ of edges is convex  and does not intersect $D$.
Moreover, directions of vectors $\xi_0$ and $\eta_j$'s satisfy the following properties.
Let us call a collection $(\zeta_1,\dots, \zeta_m)$ of $m\ge 2$ non-zero vectors in $\RR^2$
\emph{ordered counterclockwise} if $\det(\zeta_i, \zeta_j)\ge 0$ for any $i<j$.  In other words,
we call the collection counterclockwise ordered, if points $\zeta_i/||\zeta_i||$ appear on the unit
half-circle in the order prescribed by their indices.
Then the triple $(\xi_s,\xi_0,\xi_t)$ is counterclockwise ordered; moreover, since vertices $v_j$ are
numbered in the natural order of their passage along $L$, the tuple $(\xi_s,\eta_1,\dots,\eta_k,\xi_t)$
is also ordered counterclockwise. See Figure \ref{fig: pt2infty}.

\begin{figure}[htb]
\includegraphics[height=1.5in]{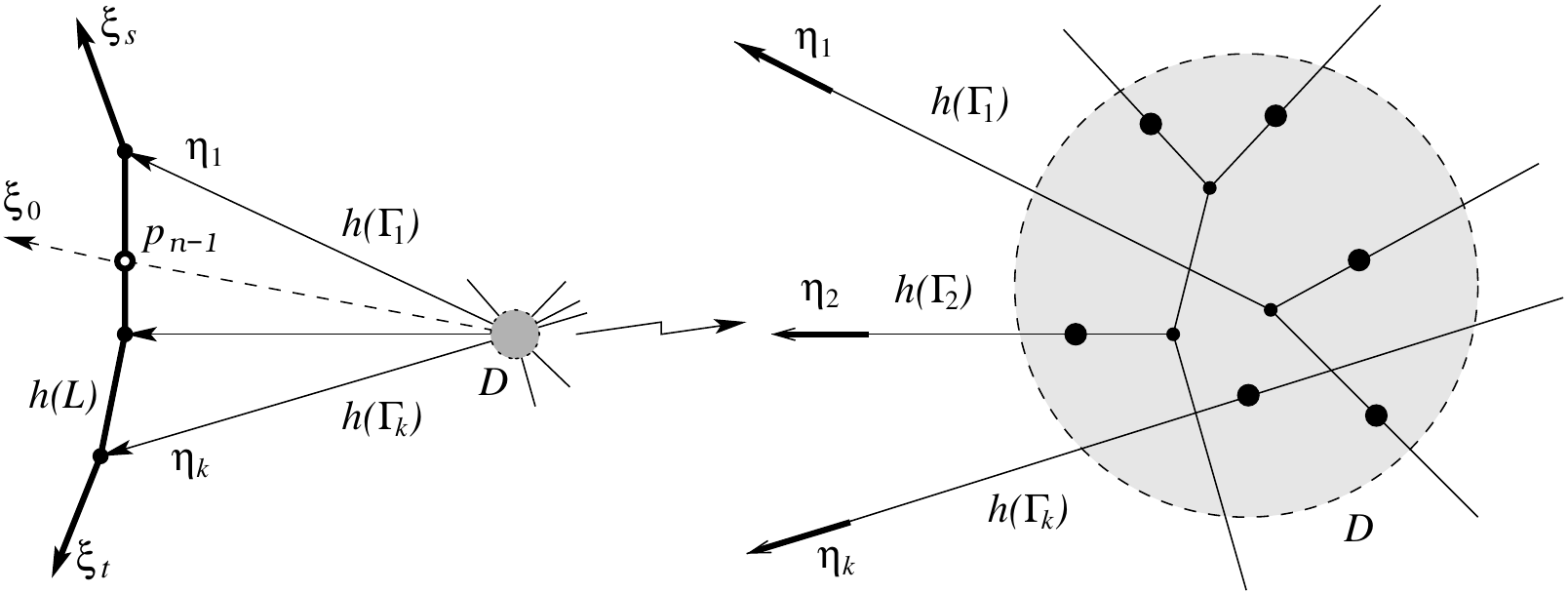}
\caption{Taking a marked point to infinity.
\label{fig: pt2infty}}
\end{figure}

It remains to find the weight $Z_{\Delta,\hbar}(\G)$ of such a curve $\G$.
Note that the slope vectors along edges of $L$ can be reconstructed directly from $S$ and $\Delta$.
Indeed, it is easy to see that the slope vector of the edge connecting $v_{j-1}$ with $v_j$ is
$-\xi_s+\eta_1+\dots+\eta_{j-1}$.
Since the other slope vector entering $v_j$ is $\eta_j$, we can easily compute the weight
$[\xi_{v_j}(e_1)\times \xi_{v_j}(e_2)]_\hbar$ of the vertex $v_j$.
To stress its dependence on the partition $S=(\{s\},\{t\},I_1,I_2,\dots,I_k)$ we denote it by $w_j(S)$):
$$
w_j(S):=\left[\sum_{i,m} \det(\xi_i,\xi_m)\right]_\hbar\, ,\quad i\in I_j\, ,\quad m\in \{s\}\cup I_1\dots\cup I_{j-1}\,.
$$
Thus we have
\begin{equation}
\label{eq: L_split}
Z_{\Delta,\hbar}(\G)=\prod_{j=1}^k w_j(S)\cdot Z_{\Delta,\hbar}(\G_j)\,.
\end{equation}
Each of the curves $C_j$ passes through a sub-collection $\mathbf{p}_j\subset \mathbf{p}$ of
$n_j$ points. Note that the number of partitions of $\{p_1,\dots,p_{n-2}\}$ into $k$ sub-collections
$\mathbf{p}_j$ of $n_j$ points is  $ \frac{(n-2)!}{n_1!n_2!\dots n_k!}$.

\subsection{The recursion formula}
We are finally ready to state a recursive formula for $N(\{\xi_1,\dots,\xi_n\})$, $n>3$ in terms of
$N(\{\xi_{i_1},\xi_{i_2},\dots,\xi_{i_m},-\eta_j\})$ for various subsets of $\xi$-vectors with $m\le n-2$.
Recall that for $n=2$ the weight $Z_{\Delta,\hbar}$ of a 1-marked line with the $\Delta$-set
$\Delta = \{\xi,-\xi\}$ is always $1$, independently on the vector $\xi$, so we have $N(\{\xi,-\xi\})=1$
for any $\xi$ (since there is one line passing through a point $p_1$ in any fixed direction $\xi$).
Thus we already know the values of $N(\xi_i,-\xi_i)=1$ and
$N(\{\xi_1,\xi_2,-\xi_1-\xi_2\})= [|\det(\xi_1,\xi_2)|]_\hbar$ for $|\Delta|=2,3$.
Thus the recursion allows one to calculate $N(\{\xi_1,\dots,\xi_n\})$ for any given $n$ and a
$\Delta$-set $\Delta = \{\xi_1,\ldots, \xi_n\}$.

\begin{defn}
Let $\Delta = \{\xi_1,\ldots, \xi_n\}$ be a $\Delta$-set. Fix a unit vector $\xi_0$.
For a subset $I=\{i_1,\dots,i_m\}$ of indices let us introduce a shorthand
notation $$N(I)=N(\{\xi_{i_1},\dots,\xi_{i_m},-\sum_{i\in I}\xi_i\})\,.$$
An (ordered) partition $S=(\{s\},\{t\},I_1,I_2,\dots,I_k)$ of $\{1,\dots,n\}$ into two
1-element subsets $\{s\}$,$\{t\}$ and $k$ non-empty subsets $I_j$ is
{\em admissible} w.r.t. $\Delta$ and $\xi_0$, if the triple $(\xi_s,\xi_0,\xi_t)$ and the tuple
$(\xi_s,\eta_1,\dots,\eta_k,\xi_t)$ are counterclockwise ordered.
\end{defn}

As we have seen in Section \ref{subsec: p_to_infinity} above, to any curve $C$ through $\mathbf{p}$
we can assign such an admissible partition and a splitting of the curve into $L$ and curves $C_j$ through
sub-collections $\mathbf{p}_j$ of marked points.
Vice versa, suppose that we are given an admissible partition $S$ with $I_j=\{i_1,\dots,i_{n_j}\}$, $j=1,\dots,k$
and an arbitrary partition of $\{p_1,\dots,p_{n-2}\}$ into $k$ sub-collections $\mathbf{p}_j$
of $n_j$ points.
For $j=1,\dots, k$ define $\Delta_j=\{\xi_{i_1},\dots,\xi_{i_n}\}$ and trace any curve $C_j$ with the
$\Delta$-set $\Delta_j$ through $\mathbf{p}_j$.
Then there is a unique curve $C$ through $\mathbf{p}$ corresponding to such a collection of curves:
we  merge $C_j$'s together with $L$ (note that the slope vectors of $L$ can be found directly from $S$
and $\Delta$, so $h(L)$ can be uniquely traced through $p_{n-1}$; admissibility of $S$ guarantees that it
is possible).
The weight of this curve is given by Equation \eqref{eq: L_split}.
Summarizing the analysis in Section \ref{subsec: p_to_infinity}, we thus conclude with

\begin{thm}
Fix a $\Delta$-set $\Delta = \{\xi_1,\ldots, \xi_n\}$ with $n>3$ and a configuration
$\mathbf{p} = (p_1,\dots, p_{n-1})\in(\RR^2)^{n-1}\minus\cD$ of $n-1$ points in
$\RR^2$ in general position.
Fix a direction $\xi_0$ in which we slide the marked point $p_{n-1}$ to infinity.
Then
$$N(\Delta) = N(\{\xi_1,\dots,\xi_n\})=\sum_{k=1}^{n-2}\sum_S \frac{(n-2)!}{n_1!n_2!\dots n_k!}
\prod_{j=1}^k w_j(S)\cdot N(I_j)$$
Here the sum is over all admissible $(k+2)$-partitions $S=(\{s\},\{t\},I_1,\dots,I_k)$
of $\{1,\dots,n\}$ as above and $n_j=|I_j|$, $j=1,\dots, k$.
\end{thm}

%\nocite{*}

\end{document}